\documentclass[11pt]{amsart}
\usepackage{geometry,fullpage} 
\usepackage{graphicx}
 \usepackage{mathptmx}      
\usepackage{latexsym}
\usepackage{mathtools}
\usepackage{amssymb}
\usepackage{enumerate}
\usepackage{color}
\usepackage{hyperref}
\usepackage[linesnumbered,ruled,vlined]{algorithm2e}
\newcommand{\dt}{\,dt}
\newcommand{\ds}{\,ds}
\newcommand{\J}{\mathrm{J}}
\newcommand{\F}{\mathrm{\bf F}}

\renewcommand{\d}{\mathrm{\bf d}}
\newcommand{\f}{\mathrm{\bf f}}
\newcommand{\g}{\mathrm{\bf g}}

\renewcommand{\O}{\mathcal{O}}
\newcommand{\U}{\mathcal{U}}
\newcommand{\V}{\mathcal{V}}

\newcommand{\R}{\mathrm{R}}
\newcommand{\N}{\mathrm{\bf N}}
\newcommand{\y}{{ \bf y}}
\newcommand{\p}{{ \bf p}}
\renewcommand{\v}{{\bf v}}
\newcommand{\w}{{\bf w}}

\newcommand{\du}{\h}
\newcommand{\h}{{\bf h}}
\renewcommand{\u}{{\bf u}}

\newtheorem{theorem}{Theorem}
\newtheorem{lemma}{Lemma}
\newtheorem{corrollary}{Corrollary}
\newtheorem{proposition}{Proposition}

\begin{document}

\title{Linearly convergent nonlinear conjugate gradient methods for a parameter identification problems}


\author[1]{Mohamed Kamel Riahi$^{\star,\dagger}$}
\address[1]{ $\dagger$
	 Department of Mathematics, Khalifa University of Sciences and Technology,\\
		PO Box 127788, Abu Dhabi, United Arab Emirates.\&\\
         Department of Mathematics, New York University in Abu Dhabi, Saadiyat Island, \\
             P.O. Box 129188, Abu Dhabi, United Arab Emirates.}
              \email{mohamed.riahi@kustar.ac.ae}
\author[2]{Issam Al Qattan$^\ddagger$}

\address[2]{
$\ddagger$ Issam. Al. Qattan 
              Department of Physics, Khalifa University of Sciences and Technology, \\
              PO Box 127788, Abu Dhabi, United Arab Emirates.}
 \email{issam.qattan@kustar.ac.ae}


\date{Received: date / Accepted: date}

\maketitle
\begin{abstract}
This paper presents a general description of a parameter estimation inverse problem for systems governed by nonlinear differential equations. The inverse problem is presented using optimal control tools with state constraints, where the minimization process is based on a first-order optimization technique such as adaptive monotony-backtracking steepest descent technique and nonlinear conjugate gradient methods satisfying strong Wolfe conditions. Global convergence theory of both methods is rigorously established where new linear convergence rates have been reported. Indeed, for the nonlinear non-convex optimization we show that under the Lipschitz-continuous condition of the gradient of the objective function we have a linear convergence rate toward a stationary point. Furthermore, nonlinear conjugate gradient method has also been shown to be linearly convergent toward stationary points where the second derivative of the objective function is bounded. 
The convergence analysis in this work has been established in a general nonlinear non-convex optimization under constraints framework where the considered time-dependent model could whether be a system of coupled ordinary differential equations or partial differential equations. Numerical evidence on a selection of popular nonlinear models is presented to support the theoretical results.
\keywords{Nonlinear Conjugate gradient methods, Nonlinear Optimal control \and Convergence analysis \and Dynamical systems \and Parameter estimation \and Inverse problem}
\end{abstract}

\section{Introduction}

Linear and nonlinear dynamical systems are popular approaches to model the behavior of complex systems such as neural network, biological systems and physical phenomena. These models often need to be tailored to experimental data through optimization of parameters involved in the mathematical formulations. Parameter estimations technique have been successfully used in a large spectrum of dynamical models, ranging from macro-scale modeling such as fluid-mechanics~\cite{gelman2014bayesian,owens2002computational,tarantola2005inverse}, aerospace and kinematics, to a micro-scale modeling such as neuron science~\cite{che2012parameter}, biological \cite{lillacci2010parameter}, semiconductors~\cite{Leitao2007,zou2004semiconductor} and chemical applications.  Many are the numerical methods that have been developed through the years in order to enhance the existing commonly used techniques and also to identifying complex dynamical systems.
 A large variety of methods are applied to solve practical problems in parameter optimization, starting from deterministic calculus~\cite{che2012parameter} methods and ranging to stochastic approaches~\cite{gelman2014bayesian}  passing by statistical approaches.  A classical technique based on a least square minimization is still widely applied in the parameter estimation and inverse problems~\cite{lee1999optimization,lillacci2010parameter,tarantola2005inverse,zou2004semiconductor} without being exhaustive. 
Among many techniques historically used we have: Newton, Levenberg-Marquardt, trust region methods. These methods and many other variants~\cite{schittkowski2013numerical} e.g. quasi-Newton's and truncated method, have shown super-linear and quadratic convergence when provided with an accurate first and second order information~\cite{guay1995optimization,vassiliadis1999second}. Despite their quality of local convergence, it is not guaranteed that any of these methods converges to a global minimizer, rather than converging to the stationary point closest to the initial guess. The weakness of the numerical optimization could not, unfortunately, be overcome unless good initial guess is provided. In order to avoid this local convergence, several techniques have been proposed in the literature such as sampling and multi-start methods~\cite{mckay2000comparison}, which consider multiple runs from a sampling of the domain of the parameters. These methods could perform a good sampling of the initial guess. However, they have been shown several restrictions and limitations~\cite{moles2003parameter,rodriguez2006novel}, where faced to large uncertainty on the range of variation of the parameter the multi-start methods become inefficient and time and memory consuming. Despite this undesirable behavior, gradient descent method is still exploited in a plenty of applications and has the advantages of being adaptive, where it can be coupled with many other global techniques in order to overcome the restriction of local convergence. In addition, gradient methods could benefits from some recently developed acceleration techniques such as \cite{maday2013parareal} or \cite{riahi2016new} to overcome the slow convergence rate du to the nature of the problem.

	This work focuses on the least square method in nonlinear optimization framework and presents its convergence property with the lowest pre-assumptions possible i.e. non-convex optimization, non-linear objective function but continuous-Lipschitz. 
The situation of lack of smoothness occurs often in real life problems, where the objective function is whether differentiable or it is smooth with a very expensive second derivative. In many applications, it is just impossible to deal with the second derivative. Because of these reasons, gradient-based technique gains its reputation among others. We will present a sufficient condition to the convergence of a descent gradient method for the minimization of the least square non-convex objective function. We also will present linear convergence rate results for a class of nonlinear conjugate gradient method satisfying strong Wolfe conditions.

	We consider a class of linear time-dependent coupled systems. It is assume throughout our analysis that the dynamical system has identifiable parameters. We formulate a nonlinear optimization problem to optimally estimate these parameters through a minimization of a misfit objective function. The optimization problem is formulated in an optimal control framework where the state variable is governed by a general nonlinear dynamic. The optimality system is given for a general nonlinear dynamics. 
	 With the help of the Lagrange multiplier, we take care of the constrained state variable, solution of the dynamical model, and consider an optimal control approach to provide the optimal parameter in term of fitting the given data. The optimality system of such a problem involves both direct and adjoint resolution of the model, and, the optimization technique requires repeating these resolutions at each iteration which helps updating the parameters through a steepest descent gradient. 
	
	 The rest of the paper is organized as follows. In Section~\ref{sec:GS}, we presents the optimal control problem in a general settings where the constraint stands for a nonlinear differential equation that governs the controlled state variable. The optimality system is then derived and monotonic-backtracking algorithm is described. In Section \ref{LipAnal}, we analyze the Lipschitz property of the state variables involved in the optimality system. In Section \ref{ConvAnal}, based on the results of the previous section, we analyze the convergence of the steepest descent method for the optimal control problem using only minimal assumptions on the smoothness of the objective function gradient. We prove that we have indeed a linear convergence rate (depending on the threshold of the iterative algorithm) for the parameter estimation inverse problem. We also report a proof of linear convergence rate of a class of nonlinear conjugate gradient that satisfy strong Wolfe conditions. 
Finally, numerical illustration of the proposed method, in a selection of a well known nonlinear problem, is presented and discussed in Section \ref{NumShow}. Concluding remarks are conducted in Section \ref{Conc} with which we close this paper. 
Throughout the paper, we denote by $\|x\|_{2}$ the Euclidian norm of a given vector $x\in\mathbb{R}^n$ associated with the scalar product $x^Tx$, where $x^T$ stands for the transpose of the vector $x$.
\section{General settings}\label{sec:GS}
For a time interval $[t_i,t_f]$, with $0<t_i<t_f$, a general description of a continuous-time nonlinear dynamical system writes as follows 
\begin{equation}\label{eq1}
\dot{\y}  = \F(t,\y,\u), \qquad\text{ for } t\in [t_i,t_f].
\end{equation}
where $\y$ stands for the state variable s.t. $\y=(y_1(t),\dots,y_i(t),\dots,y_n(t))^T$ is an $n$-by-$1$ unknown function $\in \mathcal{C}^{1}([t_i,t_f])$. $\F$ is defined in some region $B\subset\mathbb{R}^{n+m+1}$. The variable $\u$ stands for the control variable that describes a set of parameters that are involved in the mathematical model. The control variable is chosen among a set of admissible variables belonging to a given space $\U\subset \mathbb{R}^{m}$. In order to ensure existence and uniqueness of the solution of \eqref{eq1}, the function $\F$ needs to be Lipschitz or continuously differentiable in $B=\mathbb{R}\times\V\times\U$. In this sense, a solution to \eqref{eq1} is unique for each control variable $\u$ and a suitable initial condition $\y(t_i)=\y_0$.
We shall use $\langle, \rangle_{\V}$ to indicate the scalar product between vector valued functions in $\V$, which induces the norm $\|\cdot\|_{\V}$.

	 In this work, we shall assume that the function $\F$ is $\xi$-smooth function that has a Lipschitz Jacobian operator $\delta\F$ with Lipschitz constant $\xi$. This is the unique condition we are imposing through our convergence study of the nonlinear optimization problem. Indeed, we are concerned with the following finding 
\begin{equation}\label{argmin}
\u_\text{opt}=\arg\min_{\u\in \U} \J(\u),
\end{equation}
where $\J$ represents a misfit objective function that we are concerned with its minimization. A typical misfit objective function reads as follows 
\begin{equation}\label{objfnct}
\J(\u) = \dfrac{\alpha}{2}\int_{t_i}^{t_f} \|\u(t)\|_{2}^{2}  \dt+ \dfrac{1}{2}\int_{t_i}^{t_f}\|\y(t,\u)-\y^{T}(t)\|_{\V}^{2} \dt
\end{equation}
with the state variable $\y$ is solution to the nonlinear equation \eqref{eq1}. In the above equation \eqref{objfnct}, the first term appearing with $\alpha$ represents a Tickonov's regularization of the control problem. The regularization parameter can be tuned up and a selection of an optimal value of $\alpha$ could be done through an L-curve study or Morozov's discrepancy principle. This kind of analysis exceeds the contents of this paper, we may refer to \cite{ito2015inverse,kazufumi2014inverse,kirsch2011introduction} and references therein for more detailed description. 

Let $\delta\y:=\partial_{2}\y(\h)$ be the derivative of the state variable $\y$ with respect to the vector $\u$ in the direction of the perturbation vector $\h$. We shall assume that the function $\F(t,\y,\u)$ is $\mathcal{C}^1(\U)$ with first derivative Lipschitz-Continuous with respect to $\u$. We have thus
\begin{eqnarray}\label{Lipschicity-of-F}
\F(t,\y,\u+\h) &=&  \F(t,\y,\u) +\delta\F(t,\y,\u;\h) +\N(t,\y,\u;\h)\notag\\
&=& \F(t,\y,\u) + \partial_{\y}\F(t,\y,\u) \delta\y+  \partial_{2}\F(t,\y,\u)\h +\N(t,\y,\u;\h)
\end{eqnarray}
with 
$$\lim_{\du\rightarrow{\bf 0}}\dfrac{\|\N(\u;\du)\|}{\|\du\|_{2}}=0$$

	In a general framework, one can not guarantee a solution to \eqref{argmin} rather than providing a local minimum value of the objective function through a construction of convergent sequence of control $(\u_{k})_k$ to a critical point $\u^\star$. This is often the case for ill-posed problem, where most of the deterministic standard optimization fall into local critical point solution to the following optimality KKT system 
	\begin{equation}\label{KKTGen}
	\begin{cases}
	\delta\J(\u) &= 0\\
	\dot{\y}  - \F(\y,\u) &= 0\\
	\dot{\p}  +\partial_{\y}\F^*(\y,\u)\p&=\y(t,\u)-\y^T(t),
	\end{cases}
	\end{equation}
 from which we can see that a necessary condition for the well posedness of the adjoint equation  \eqref{KKTGen}$_3$ is that the derivative $\delta\F$ should be at least Lipschitz continuous operator. This is indeed, our sufficient condition to prove that the optimization algorithm converges linearly. Starting from giving an expression to the first derivative of the objective function as 
 \begin{equation}\label{GradCostGen}
 \delta\J(\u;\delta\u) = \alpha\int_{t_i}^{t_f} \langle \delta\u(t),\u(t) \rangle_{2} \dt + \int_{t_i}^{t_f} \langle \delta\y(\u;\delta\u.t),\y(t,\u)-\y^T(t) \rangle_{\V} \dt
 \end{equation}
where $\delta\y$ satisfies 
\begin{equation}\label{deltay}
\begin{cases}
\dot{\delta\y} = \delta_\y\F(\y;\u)\delta\y + \partial_{2}\F(\y,\u)\delta\u, \quad t\in[t_i,t_f]\\
\delta\y(t_i) =  0,
\end{cases}
\end{equation}
and, by introducing the adjoint state variable $\p$ solution to 
\begin{equation}\label{Eqadj}
\begin{cases}
-\dot{\p}  - \partial_{\y}\F^*(\y;\u)\p =\y(t,\u)-\y^T(t), \quad t\in[t_f,t_i]\\
\p(t_f) =  0. 
\end{cases}
\end{equation}
the first derivative \eqref{GradCostGen} of the objective function writes therefore
\begin{eqnarray*}
 \delta\J(\u;\delta\u) &=& \alpha\int_{t_i}^{t_f} \langle \delta\u(t),\u(t) \rangle_{2} \dt + \int_{t_i}^{t_f} \langle \partial_{2}\F(\y,\u)\delta\u,\p(t) \rangle_{\V} \dt\\
		      	      &=& \int_{t_i}^{t_f} \langle \alpha  \u + (\partial_{2}\F(\y,\u))^*\p, \delta\u \rangle_{\V} \dt.\\
			      &=&  \int_{t_i}^{t_f} \langle \g(t,\u), \delta\u\rangle_{2} \dt
\end{eqnarray*}
where the gradient of the objective function writes
\begin{equation}\label{gradient}
\g(t,\u) = \alpha  \u + (\partial_{2}\F(\y,\u))^*\p.
\end{equation}
Once the above explicit expression of the gradient \eqref{gradient} is provided. we can proceed with the minimization of the objective function, following the classical descent gradient approach, as stated in the below {\bf Algorithm~\ref{algo1}}
\begin{center}
\begin{algorithm}[H]
\DontPrintSemicolon 
\KwIn{Initial guess $\u^0$, initial condition $\y_0$, Tolerance $\epsilon$.\;}
\KwOut{$\u^\star$ stationary point}
\While{$\|\g_{k}\|_{2}^{2}\geq \varepsilon$} {
	\text{Solve Forward problem} for $\y_{k}$ using \eqref{KKTGen}$_2$ \;
	\text{Solve Backward problem} for $\p_{k}$ using \eqref{KKTGen}$_3$ \;
	\text{Evaluate the gradient} $\g_{k}$ using \eqref{gradient}\;
    	$\u_{k+1} \gets \u_{k}-\frac{1}{2\xi}\g_{k}$ \;
    	$k\gets k+1$\;
}
\Return{$\u_{k+1}$}\;
\caption{\sc Parameter estimations iterative algorithm}
\label{algo1}
\end{algorithm}
\end{center}
Provided with an initial guess $\u^0$, {\bf Algorithm \ref{algo1}} converges properly, but slowly, to the closed critical point. Without a prior knowledge on the Lipschicity of the handled function, it is often hard to feed this algorithm with the Lipschitz constant $\xi$, in this situation the step-length $\frac{1}{2\xi}$ is chosen to be small enough to ensure the minimization process. A monotony-backtracking procedure has been proven to be efficient in this situations. We provide in {\bf Algorithm \ref{algo2}} such technique for such non-linear non-convex optimization framework.   

	 Our analysis relies mainly on the Lipschitz property of the objective function's gradient used for the optimal control problem. We can see from \eqref{gradient} that a Lipschitz property of both $(\partial_{2}\F(\y,\u))^*$ and $\p$ is needed. Actually, following the assumption \eqref{Lipschicity-of-F} we have just to provide Lipschicity constant for the adjoint state $\p(\u)$, which, we recall it, is function of the parameter $\u$ through out the state variable $\y(\u)$.


\section{Lipschicity analysis}\label{LipAnal}

This section, is devoted to rigorously provide the well-posedness of the optimality condition system \eqref{KKTGen} that includes the explicit formula for the objective function gradient, the state variable and the adjoint variable. We shall use necessary conditions, to come up with a Lipschitz property of objective function gradient.
We need to prove that $\y$ is $F$-differentiable with respect to the control variable $\u$. This helps us proving the $F$-differentiability of the objective function. As we have formally shown in the above section, the introduction of the adjoint state variable is necessary to give an explicit mathematical formula of the gradient of the objective function $\J$. Once introduced, we need to provide that the adjoint state variable $\p$ in its turn is Lipschitz with respect to $\u$. 

\begin{proposition} Assume that the source term $\F(t,\y,\u)$ is $F$-differentiable with Lipschitz-continuous first derivative $\delta\F(t,\y,\u)$, this implies that the state variable $\y$ solution to \eqref{eq1} is $F$-differentiable with bounded first derivative, henceforth Lipschitz.
\end{proposition}
\begin{proof}
consider two different controls variables $\u$ and $\v:=\u+\h$, and call the solution $\y{(t,\u)}$ respectively $\y{(t,\v)}$ the controlled solution associated to the control $\u$ respectively to the control $\v$. Assuming the same initial condition for both solution $\y(\u+\h;0)=\y(\u;0)$. We have 
\begin{eqnarray*}
\dot{\y}{(\u+\h)} &=& \F(t,\y,\u+\h),\\
\dot{\y}{(\u)}      &=& \F(t,\y,\u),
\end{eqnarray*}

taking the difference of these two equations we obtain:
\begin{eqnarray*}
\big(\dot{\y}{(\u+\h)} -\dot{\y}{(\u)} \big)&=& \F(t,\y,\u+\h) -  \F(t,\y,\u)\\
&=& \delta\F(t,\y,\u;\h) +\N(t,\y,\u;\h).
\end{eqnarray*}
Since the right hand side consists in two different contributions that linearly affect the solution $\big(\dot{\y}{(\u+\h)} -\dot{\y}{(\u)} \big)$. The former, could in its turn be seen as a two contributions of two different solutions $\delta\y$ and $\y_\N$. These new variables are solutions to the following two equations
\begin{eqnarray*}
\dot{\delta\y}&=&  \delta\F(t,\y,\u;\h), \\
\dot{\y}_\N &=&  \N(t,\y,\u;\h).
\end{eqnarray*}
Supplemented with the initial condition at $t=0$, $\delta\y(0;\h)=0$ and $\y_\N(0;\h)=0$. More concretely we have 
\begin{equation}\label{dletay}
\dot{\delta\y}  =   \partial_{\y}\F(t,\y,\u) \delta\y+  \partial_{2}\F(t,\y,\u)\h 
\end{equation}
while the second variable $\y_\N$ is driven by the nonlinear Lipschitz-continuous operator $\N$ (in fact it is the difference of two Lipschitz-continuous operators). This implies that the solution $\y_\N$ exists. Furthermore, since the source nonlinear term  $\N(t,\y,\u;\h)\approx\O(\|\h\|_{2}^{2})$ vanishes as $\|\h\|_{2}$ approaches zero, this immediately implies that $\y_\N$ in its turn approaches zero as $\|\h\|_{2}$ approaches zero. Finally we have  
\begin{equation}
\y(\u+\h)-\y(\u) = \delta\y(\u;\h) + \y_\N(\u;\h)
\end{equation}
with 
$$\lim_{\h\rightarrow 0} \dfrac{\|\y_\N(\u;\h)\|_{\V}}{\|\h\|_{2}}=\lim_{\h\rightarrow 0} \O(\|\h\|_{2})=0.$$
Therefore the $F$-differentiability with respect to the control variable $\u$ of the state variable $\y$ solution to the nonlinear equation \eqref{eq1}. 
	
	Let $\varphi$ be the fundamental matrix of the equation $\dot{\delta\y} = \partial_{\y}\F(t,\y,\u) \delta\y$, satisfying $\varphi(t_i)=I_n$. It is clear that under the aforementioned assumptions, the operator $\partial_{\y}\F(t,\y,\u)$ is bounded on the interval $[t_i,t_f]$, which implies that all solution of \eqref{dletay} are bounded thus uniformly stable. Furthermore, there exist a positive constant $K_\delta$ for which we have 
$$\|\varphi(t)\varphi(s)^{-1}\|_\V\leq K_\delta, \quad \text{ for }\,  t_i \leq t \leq t_f.$$
In addition, a solution to \eqref{dletay} writes 
\begin{equation}
\delta\y(t,\u)  = \int_{t_i}^{t} \varphi(t) \varphi^{-1}(s)  \partial_{2}\F(s,\y(s),\u(s))\h(s) \ds,  \text{ for } t_i \leq t \leq t_f
\end{equation}
Therefore the solution $\delta\y(t,\u)$ is easily proven to be bounded as $\partial_{2}\F(t,\y(s),\u(s))$ is. Also, the solution $\y_\N(t,\u;\h)=\int_{t_i}^{t}\N(s,\y,\u,\h)\ds$ is bounded for any $t\in[t_i,t_f]$.  We have for $\h=\v-\u$
\begin{eqnarray*}
\|\y(\v)-\y(\u)\|_\V &=& \|\delta\y(\u;\v-\u) + \y_\N(t,\u,\v-\u)\|_\V\\
&\leq& \|\delta\y(\u;\v-\u)\|_\V + \| \y_\N(t,\u,\v-\u)\|_\V\\
&\leq& K_\delta L  \|\v-\u\|_{2} + N \|\v-\u\|_{2}\\
&=& (K_\delta L+N)  \|\v-\u\|_{2}.
\end{eqnarray*}
The proof is complete.$\hfill\qed$
\end{proof}

\begin{corrollary}[$F$-differentiability of $\p$] 
The adjoint state $\p$ is $F$-differentiable with respect to $\u$. 
\end{corrollary}
\begin{proof}
The adjoint state variable $\p$ is solution to the linear equation \eqref{Eqadj} and writes
$$
\p(t,\u) = \int_{t_i}^{t_f} \varphi^{*}(t)\varphi^{-*} (s) \left( \y(s,\u)-\y^T(s)\right) \ds
$$

\begin{eqnarray*}
\|\p(t,\v)-\p(t,\u)\|_{\V} &=& \big\|\int_{t_i}^{t_f} \varphi^{*}(t)\varphi^{-*} (s) \left( \y(s,\v)-\y(s,\u)\right) \ds\big\|_{\V} \\
&\leq& |t_f-t_i| K_\p \|\y(s,\v)-\y(s,\u)\|_{\V}\\
&\leq& |t_f-t_i| K_\p (K_\delta L+N)  \|\v-\u\|_{2}.
\end{eqnarray*}
The proof is complete.$\hfill\qed$
\end{proof}

\begin{theorem}
The objective function $\J$ is $F$-differentiable with respect to the control variable $\u$ where we have 
\begin{equation}
\J(\u+\h)-\J(\u) = \delta\J(\u;\h) + \O(\|\h\|_{2})
\end{equation}
\end{theorem}

\begin{proof}
This proof relies on the $F$-differentiability of the state variable $\y$. 
\begin{eqnarray*}
\J(\u+\h) &=&  \dfrac{\alpha}{2} \|\u+\h\|_{2}^{2} + \dfrac{1}{2}\int_{t_i}^{t_f}\|\y(t,\u+\h)-\y^{T}(t)\|_{\V}^{2}\\
&=&  \dfrac{\alpha}{2} \|\u\|_{2}^{2} +  \dfrac{\alpha}{2} \|\h\|_{2}^{2}  + \alpha \langle \u,\h\rangle_{2}   \\
&&+ \dfrac{1}{2}\int_{t_i}^{t_f}\|\y(t,\u)+\delta\y(t,\u;\h) + \y_\N(t,\u;\h)-\y^{T}(t)\|_{\V}^{2}\\
&=& \dfrac{\alpha}{2} \|\u\|_{2}^{2} + \dfrac{1}{2}\int_{t_i}^{t_f}\|\y(t,\u)-\y^{T}(t)\|_{\V}^{2} +  \dfrac{\alpha}{2} \|\h\|_{2}^{2}  + \alpha \langle \u,\h\rangle_{2}   \\
&&+ \dfrac{1}{2}\int_{t_i}^{t_f}\|\delta\y(t,\u;\h) + \y_\N(t,\u;\h)\|_{\V}^{2}\\
&&+ \int_{t_i}^{t_f}\langle \delta\y(t,\u;\h) + \y_\N(t,\u;\h),\y(t,\u)-\y^{T}(t)\rangle_{\V}\\
&=& \J(\u) + \alpha \langle \u,\h\rangle_{2} + \int_{t_i}^{t_f}\langle \delta\y(t,\u;\h) ,\y(t,\u)-\y^{T}(t)\rangle_{\V}\\
&&+ \int_{t_i}^{t_f}\langle  \y_\N(t,\u;\h),\y(t,\u;\h)-\y^{T}(t)\rangle_{\V}\\
&&+\dfrac{\alpha}{2} \|\h\|_{2}^{2} + \dfrac{1}{2}\int_{t_i}^{t_f}\|\delta\y(t,\u;\h) + \y_\N(t,\u;\h)\|_{\V}^{2}\\
&=& \J(\u) + \delta\J(\u;\h)+ \int_{t_i}^{t_f}\langle  \y_\N(t,\u;\h),\y(t,\u)-\y^{T}(t)\rangle_{\V}\\
&&+\dfrac{\alpha}{2} \|\h\|_{2}^{2} + \dfrac{1}{2}\int_{t_i}^{t_f}\|\delta\y(t,\u;\h) + \y_\N(t,\u;\h)\|_{\V}^{2}
\end{eqnarray*}

\begin{eqnarray*}
\left|\J(\u+\h)-\J(\u) - \delta\J(\u;\h) \right| &\leq& \|\y_R(\u;\h)\|_{\V}\|\y(t,\u;\h)-\y^{T}(t)\|_{\V}\\
&& \dfrac \alpha 2 \|\h\|_{2}^{2} + \|\delta \y_{L}(t,\u;\h)\|_{\V} \|\delta \y_{R}(t,\u;\h)\|_{\V}\\
&&\leq \O(\|\h\|_{2}).
\end{eqnarray*}
The proof is complete.$\hfill\qed$
\end{proof}

\section{Convergence Analysis}\label{ConvAnal}
We shall present in this section, rate of convergence results related to the steepest descent method and the nonlinear conjugate gradient methods satisfying strong Wolfe conditions. 
At the first stage, we assume that the objective function is $\xi$-smooth function i.e. has a  Lipschitz gradient. Without any further assumptions the convergence of the steepest descent method could be proven to be linear with a rate lying between half and one. The claimed rate could certainly be improved  once additional properties of the objective function are given. 
Beside, in the second stage a nonlinear conjugate method is then considered. Satisfying strong Wolfe conditions NCG methods are shown to be linearly convergent is the second derivative of the objective function is bounded. 

\subsection{Rate of convergence for the gradient descent method}
	For the steepest descent method, we will restrict our selfs in the necessary conditions (that ensures existence of the solution) on the dynamical model, and the following results holds.
\begin{theorem} 
Algorithm \ref{algo1} converges linearly with rate $\zeta$ satisfying $\dfrac 1 2 < \zeta < 1$,when minimizing a $\xi$-smooth objective function
\end{theorem}
Before we start the proof, let us define the shifted objective function 
$$
\tilde\J_{k}  = \J_{k} - \J(\u^\star),
$$
which will be useful in the sequel.
\begin{proof}
Thanks to Taylor theorem we have
\begin{eqnarray*}
\J_{k+1}  &=& \J(\u_{k}-t\g_{k})   \\
&=& \J_{k} - t\|\g_{k}\|_{2}^{2} -t \int_{0}^{1} \langle \g\big(\u_{k}-st\g_{k}\big)-\g_{k},\g_{k} \rangle_{2} \ds \\
&\leq& \J_{k} - t\|\g_{k}\|_{2}^{2}  +t \|\g_{k} \| \int_{0}^{1} \| \g\big(\u_{k}-st\g_{k}\big)-\g_{k}\|_{2}  \ds \\
&\leq& \J_{k} - t \|\g_{k}\|_{2}^{2}  + \dfrac{t^2\xi}{2}\|\g_{k}\|_{2}^{2} \\
&\leq& \J_{k} + t \left(\dfrac{t\xi}{2}-1\right)\|\g_{k}\|_{2}^{2},
\end{eqnarray*}
for which any $0<t\leq \frac 2 \xi$ ensures the monotony of the sequence $\left(\J_{k}\right)_k$. For simplicity, we shall fix in the sequel $t=\frac{1}{\xi}$ to obtain 
\begin{equation}\label{J-Monotony}
\dfrac{1}{2\xi}\|\g_{k}\|_{2}^{2} \leq \J_{k}-\J_{k+1} 
\end{equation}
Summing up the first $k$ iteration in \eqref{J-Monotony} we obtain 
\begin{eqnarray*}
\dfrac{1}{2\xi} \sum_{\ell=0}^{k-1}  \|\g_{\ell}\|_{2}^{2} &\leq& \J(\u^{0}) - \J_{k} \\
&\leq& \J(\u^{0}) - \J(\u^{\star}) \\
&:=& \tilde\J(\u^{0})\\& =& \tilde\J_{0}.
\end{eqnarray*}
In addition, because of $\left(\J_{k}\right)_k$ is a convergent Cauchy sequence, the above inequality holds true as well for an infinite sum. In particular we have 
\begin{equation}\label{normGradBndedBlw}
 \sum_{\ell=0}^{\infty}  \|\g_{\ell}\|_{2}^{2}    \leq   2\xi \tilde\J(\u^{0}) 
\end{equation}
In the other hand, we have
\begin{eqnarray}
\tilde\J_{k} - \tilde\J_{k+1} &=& \dfrac{1}{2\xi} \int_{0}^{1} \langle\g(\u_{k} - \dfrac{s}{2\xi} \g_{k}),\g_{k} \rangle_{2} \ds\notag\\
&=& \dfrac{1}{2\xi}  \langle\g(\u_{k} - \dfrac{\tau}{2\xi} \g_{k}),\g_{k} \rangle_{2}\notag\\
&\leq& \dfrac{1}{2\xi} \|\g_{k}\|_{2} \|\g(\tilde\u_{k})\|_{2} \label{difflow}.
\end{eqnarray}
Where we have used the mean value theorem for the second inequality and considered $\tilde\u_{k}=\u_{k} - \dfrac{\tau}{2\xi} \g_{k}$ in the third inequality after using Cauchy-Schwartz. It is worth recalling that in this interpolation $\tau\in(0,1)$. Thanks to the fact that the gradient is a Lipschitz-continuous function and the continuity of the norm inequality, we have
\begin{eqnarray*}
\bigg| \|\g(\tilde\u_{k})\|_{2} - \|\g_{k}\|_{2} \bigg| &\leq& \|\g(\tilde\u_{k})- \g_{k}\|_{2}\\
&\leq & \xi \|\tilde\u_{k}-\u_{k}\|_{2}\\
&\leq & \dfrac{\tau}{2} \|\g_{k}\|_{2}
\end{eqnarray*}
Therefore, 
$$
\|\g(\tilde\u_{k})\|_{2} \leq \left(\dfrac{\tau}{2}+1\right) \|\g_{k}\|_{2} \leq 2 \|\g_{k}\|_{2}
$$
Henceforth, the inequality \eqref{difflow} becomes 
\begin{equation*}
\tilde\J_{k} - \tilde\J_{k+1} \leq \dfrac{1}{\xi} \|\g_{k}\|_{2}^{2},
\end{equation*}
which leads, after summing up terms, to 
\begin{eqnarray}
\tilde\J_{k}  &\leq& \dfrac{1}{\xi} \sum_{\ell\geq k}  \|\g_{\ell}\|_{2}^{2}\notag\\
&\leq& \dfrac{1}{\xi} \|\g_{k}\|_{2}^{2} \left( 1+ \dfrac{\sum_{\ell\geq k+1}  \|\g_{\ell}\|_{2}^{2}}{\|\g_{k}\|_{2}^{2}} \right)\label{upperJ_0}\\
&\leq& \dfrac{1+ 2\xi\tilde\J(\u^0)/\varepsilon_{k}}{\xi} \|\g_{k}\|_{2}^{2} \label{upperJ_1}
\end{eqnarray}
In order to prove \eqref{upperJ_1} we have used \eqref{normGradBndedBlw} together with the fact that before convergence we have $\|\g_{k}\|_{2}^{2}\geq \varepsilon_{k}$ in \eqref{upperJ_0}, where $\varepsilon_{k}$ is any sequence that converges asymptotically to the zero. Note that we always can find such a non-necessarily vanishing sequence that lower bound the length of the gradient and asymptotically equivalent to $\|\g_k\|_{2}$.

In its turn \eqref{upperJ_1} gives 
\begin{equation}\label{SqrtUperBnd}
\sqrt{\tilde\J_{k} } \leq  \sqrt{\dfrac{\varepsilon_{k}+2\xi\tilde\J(\u^0)}{\varepsilon_{k}\xi}} \|\g_{k}\|_{2}
\end{equation}
Henceforth, 
\begin{equation}\label{InvSqrtJ}
\dfrac{1}{\sqrt{\tilde\J_{k}}} \geq \sqrt{\dfrac{\varepsilon_{k}\xi}{\varepsilon_{k}+2\xi\tilde\J(\u^0)}} \dfrac{1}{ \|\g_{k}\|_{2}}
\end{equation} 
	Furthermore, since $\sqrt{x}$ is a concave function then it is bounded above by its first order Taylor expansion. Indeed, we have
$$
\sqrt{\tilde\J_{k+1}} \leq \sqrt{\tilde\J_{k}} + \dfrac{\J_{k+1}-\J_{k}}{2\sqrt{\tilde\J_{k}}} 
$$
then, using \eqref{J-Monotony} and \eqref{InvSqrtJ}  we obtain
\begin{equation*}
\sqrt{\tilde\J(\u_{k} )} - \sqrt{\tilde\J_{k+1}} \geq  \dfrac{1}{2\xi} \sqrt{\dfrac{\varepsilon_{k}\xi}{\varepsilon+2\xi\tilde\J(\u^0)}} \|\g_{k}\|_{2}.
\end{equation*}
which we sum up for $\ell\geq k$ to have
\begin{equation}\label{SqrtLwerBnd}
\sqrt{\tilde\J(\u_{k} )}  \geq  \dfrac{1}{2\xi} \sqrt{\dfrac{\varepsilon_{k}\xi}{\varepsilon_{k}+2\xi\tilde\J(\u^0)}} \sum_{\ell\geq k}^{}\|\g_{\ell}\|_{2}.
\end{equation}
Now, combining \eqref{SqrtUperBnd} and \eqref{SqrtLwerBnd} we have
\begin{equation}\label{PlainTheInequality}
  \dfrac{1}{2\xi} \sqrt{\dfrac{\varepsilon_{k}\xi}{\varepsilon_{k}+2\xi\tilde\J(\u^0)}} \sum_{\ell\geq k}^{}\|\g_{\ell}\|_{2} 
  \leq \sqrt{\dfrac{\varepsilon_{k}+2\xi\tilde\J(\u^0)}{\varepsilon_{k}\xi}} \|\g_{k}\|_{2}
\end{equation}
By setting $\gamma_{k}= \sqrt{\dfrac{\varepsilon_{k}\xi}{\varepsilon_{k}+2\xi\tilde\J(\u^0)}}$, and $\R_k=\sum_{\ell\geq k}^{}\|\g_{\ell}\|_{2}$, \eqref{PlainTheInequality} becomes 
\begin{equation*}
  \dfrac{\gamma_{k}}{2\xi} R_k
  \leq \dfrac{1}{\gamma_{k}} \left(R_{k}-R_{k+1}\right)
\end{equation*}
Therefore
\begin{eqnarray*}
R_{k+1} &\leq& \left(\dfrac{2\xi}{\gamma_{k}^2}-1 \right) \big\slash \left( \dfrac{2\xi}{\gamma_{k}^2}\right)\\
 &\leq& \left( \dfrac{2\xi -\gamma_{k}^2}{2\xi} \right) R_k\\
  &\leq& \left( \dfrac{2\xi -\gamma_{k}^2}{2\xi} \right)^{k} R_0\\
    &=& \zeta_{k} R_0\\
\end{eqnarray*}
It is then clear that the rate $\zeta\in(\dfrac 12,1)$, which ends the proof.$\hfill\qed$ 
\end{proof}

\begin{algorithm}[!htbp]
\DontPrintSemicolon 
\KwIn{Initial guess $\u^0$, initial condition $\y_0$, steplength $\alpha_{\bullet}$, maximum iterations $k_\text{max}$}
\KwOut{$\u^\star$ stationary point}
$k \gets 1$\;
$\text{Flag} \gets True$\;
	\text{Solve Forward problem} for $\y_{k}$ using \eqref{KKTGen}$_2$ \;
	\text{Solve Backward problem} for $\p_{k}$ using \eqref{KKTGen}$_3$ \;
	\text{Evaluate the gradient} $\g_{k}$ using \eqref{gradient}\;	
\While{$k <k_\text{max} \,\&\& \,\|\g_{k}\|_{2}^{2}\geq \varepsilon$} {
\uIf{$Flag$}{
	$\alpha \gets \alpha_{\bullet}$\;
	\text{Evaluate the gradient} $\g_{k}$ using \eqref{gradient}\;
	$\u_\text{new} \gets \u_{k} + \alpha \g_{k}$\;
	$k_{\bullet}\gets k$\;
    }
    \Else{
    $\alpha\gets\alpha/2$\;
	$\u_\text{new} \gets \u_{k} + \alpha \g_{k}$    
    }
$k\gets k_{\bullet} + 1$\;
$\u_{k+1}\gets \u_\text{new}$\;
	\text{Solve Forward problem} for $\y_{k}$ using \eqref{KKTGen}$_2$ \;
	\text{Solve Backward problem} for $\p_{k}$ using \eqref{KKTGen}$_3$ \;
$Flag \gets \text{logical}(\J(k)<\J(k_{\bullet}))$\;
}
\Return{$\u_{k}$};\;
\caption{{\sc Parameter estimations adaptive monotony-backtracking algorithm}}
\label{algo2}
\end{algorithm}

In Algorithm 2, we present an enhanced version of Algorithm 1, where a monotony-backtracking based approach is implemented. Indeed, in order to make sure that the objective functional gets decreasing throughout the iterations, we adjust the step length of the steepest gradient descent to be smaller as necessary to ensure the monotony of the optimization. This is a sort of dummy line search, although, it guarantees convergence of the gradient method and avoid any possible cancelation nearby the stationary point if the step-length has been badly chosen initially. 
\subsection{Rate of convergence for a class of nonlinear conjugate gradient methods with inexact line search}

The nonlinear conjugate gradient (NCG) method applies to a problem of minimization of nonlinear nonquadratic real-valued functions. Usually, there are two ways which the NCG can be used; the "continued" method and the "restarted" method. In the later,  after every $n$ iterations, all data except the best previous point are discarded and the new iterations restart all over again from that point, hence rebuild a new sequence of conjugate directions. In practice, it has been generally proven that the restarting NCG method performs better than the continued method. Actually, in \cite{cohen1972rate} it has been shown through examples that the continued method has convergence rate at worst linear, while a quadratic rate of convergence might be achieved with the restarted method \cite{cohen1972rate,mccormick1974alternative}. We refer to \cite{hager2006survey,Dai2011} for a recent survey on the global convergence results related to different NCG methods previously and recently published. 

	In this work, our effort focuses on the continued version of the NCG and provides linear convergence rate for the majority of a classical well-known methods.
	
	In general context, conjugate gradient methods aim at minimizing a given objective function, say $\J(\u_{})$, by updating the variable $\u_{k}$ as follow 
\begin{equation}\label{NCG_update}
\u_{k+1}=\u_{k} + \alpha_{k} \d_{k},
\end{equation}
where at a given iteration $k$, $\alpha_{k}>0$ stands for the step-length that needs to be determined along the descent search direction $\d_{k}$ defined by 
\begin{eqnarray}\label{NCG_direction}
\d_{k} = \left\{\begin{array}{lr}
-\g_{k}, & \text{for } k=1\\
-\g_{k} + \beta_{k}\d_{k-1}, & \text{for } k\geq2
\end{array}\right.
\end{eqnarray}
with $\g_{k}=\nabla\J(\u_{k})$ and $\beta_{k}>0$ is a parameter, with which we distinguish a NCG method from another.  The first attempt to extend the linear conjugate gradient (from the quadratic minimization problem to a fully nonlinear) starts with \cite{fletcher1964function}.

Some well known formulas for $\beta_{k}$ are given by the Fletcher-Reeves (FR) method \cite{fletcher1964function}, Polak-Ribi\`ere \cite{pola1969note}, Hestenes-Stiefel (HS) method \cite{hestenes1952methods}, and Dai-Yuan (DY) method \cite{dai1999nonlinear}. These methods define $\beta_{k}$ by 
\begin{eqnarray}
\beta_{k}^{FR} &=& \|\g_{k}\|_{2}^{2}  / \|\g_{k-1}\|_{2}^{2} \label{NCGMFR}\\
\beta_{k}^{PR} &=&\g_{k}^T \triangle\g_{k-1} / \|\g_{k-1}\|_{2}^{2}\label{NCGMPR}\\
\beta_{k}^{HS} &=&\g_{k}^T \triangle\g_{k-1} / \d_{k}^T\triangle\g_{k-1}\label{NCGMHS}\\
\beta_{k}^{DY} &=&\|\g_{k}\|_{2}^{2} / \d_{k}^T\triangle\g_{k-1}\label{NCGMDY},  
\end{eqnarray}
where $\triangle\g_{k} = \g_{k}-\g_{k-1}$. 

The global convergence properties of the above methods in their continued version (i.e. without restarts) have been investigated by many authors, such that Zoutendijk \cite{zoutendijk1970nonlinear}, Al- Baali \cite{al1985descent}, Liu, Han, and Yin \cite{guanghui1995global}, Dai and Yuan \cite{dai1996convergence}, Powell \cite{powell1984nonconvex}, Gilbert and Nocedal \cite{gilbert1992global}, and Dai and Yuan \cite{dai1995further}. To establish the convergence results of these methods, it is normally required that the step-length $\alpha_{k}$ satisfy the following strong Wolfe conditions (Fletcher's and Goldstein requirement) respectivelly
\begin{eqnarray}\label{WolfCond}
\J_{k+1} - \J_{k} &\leq \rho \alpha_{k} \g_{k}^T \d_{k}\label{Fletcher}\\
\left|\g_{k+1}^T\d_{k} \right| &\leq -\sigma \g_{k}^T\d_{k}\label{Goldstein}
\end{eqnarray}
where $\sigma\in(0,\frac 1 2]$ and $0\leq \rho\leq\sigma$. The Goldstein requirement \eqref{Goldstein} is often regarded as a relaxed extension of the exact line search since it reduces to the later if $\sigma$ vanishes. Equation \eqref{Goldstein} ensures, indeed, that the modulus of the slope is reduced by a factor of $\sigma$ or less through the line search.

	For our analysis let us recall the following results
\begin{theorem}\cite[Al-Baali: Theorem 1]{al1985descent}\label{AlBaali}
If $\alpha_{k}$ satisfies \eqref{Fletcher}-\eqref{Goldstein} with $\sigma\in(0,\frac 1 2]$ for all k ($\g_{k}\neq 0$), then the descent property for the nonlinear conjugate gradient \eqref{NCG_update}-\eqref{NCG_direction} (s.t. \eqref{NCGMFR}-\eqref{NCGMDY}) method holds for all such $k$, more precisely 
\begin{equation}\label{AlBAAliTHM}
\dfrac{-1}{1-\sigma} \leq \dfrac{\g_{k}^T\d_{k}}{\|\g_{k}\|_{2}^{2}} \leq \dfrac{2\sigma-1}{1-\sigma},
\end{equation}
\end{theorem}
\begin{proof}
See \cite[Theorem 1]{al1985descent}. The proof is by induction arguments for any nonlinear conjugate gradient method \eqref{NCG_update}-\eqref{NCG_direction} that satisfies strong Wolfe conditions \eqref{Fletcher}-\eqref{Goldstein}.
\end{proof}
From \eqref{AlBAAliTHM} we can derive the following bound 
\begin{equation}\label{AL-BAALI-INEQ}
\|\g_{k}\|_{2} \leq \left(\dfrac{1-\sigma}{1-2\sigma}\right) |\g_{k}^T \d_{k}|.
\end{equation}

	Using the above equation we state the following Lemma.
\begin{lemma}\label{LemmaRIAHI}
If the descent direction $\d_{k}$ defined as \eqref{NCG_direction}  satisfies the condition \eqref{Goldstein}, we then have for any integer $\ell>k+1$
\begin{equation}\label{EqLemmaRIAHI}
\left|\g_{\ell}^T \d_{\ell} \right| \leq (\beta\sigma)^{\ell-k-1} \left|\g_{k}^T \d_{k} \right|,
\end{equation}
where $\beta = \max_{k+1\leq j\leq \ell}\beta_{j}$.
\end{lemma}
\begin{proof}
Starting from \eqref{NCG_direction} we can write 
$$\g_{k+1}^T\d_{k+1}   + \|\g_{k+1}\|_{2}^{2} = \beta_{k+1} \g_{k+1}^T\d_{k}, \quad \text{ for } k\geq 2$$
thus 
$$\left|\g_{k+1}^T\d_{k+1}   + \|\g_{k+1}\|_{2}^{2}  \right|= \left| \beta_{k+1} \g_{k+1}^T\d_{k} \right|\leq \beta_{k}\sigma |\g_{k}^T\d_{k}|,$$
having used \eqref{Goldstein}. Therefore 
$$
\dfrac{\left|\g_{k+1}^T\d_{k+1} \right|}{\left|\g_{k}^T\d_{k} \right|} \leq \beta_{k+1}\sigma,
$$
and the result follow using bootstrapping argument as 
$$
\dfrac{\left|\g_{\ell}^T\d_{\ell} \right|}{\left|\g_{k}^T\d_{k} \right|} \leq \sigma^{\ell-k-1} \prod_{j=k+1}^{\ell}\beta_{j}.
$$
It is finally clear that, with $\beta=\max_{j}\beta_{j}$, we have 
\begin{equation}
\dfrac{\left|\g_{\ell}^T\d_{\ell} \right|}{\left|\g_{k}^T\d_{k} \right|} \leq \left(\beta\sigma\right)^{\ell-k-1}.
\end{equation}
The proof is complete. $\hfill\qed$
\end{proof}
Equation \eqref{EqLemmaRIAHI} in the light of \eqref{AL-BAALI-INEQ} leads immediately to the following 
strong global convergence result.

It is worth noticing that the Al-Baali's proof of global convergence\cite[Theorem. 2]{al1985descent} has been derived using mainly \eqref{AlBAAliTHM} to find contradiction, where is has been supposed that $\sigma< 1/2$. Liu et al. \cite{Guanghui1995} extended the result to the case that $\sigma=1/2$ then \cite{doi:10.1093/imanum/16.2.155} simplified the proof further. The following result demonstrates strong global convergence result under certain condition. 

\begin{theorem}
If $\alpha_{k}$ satisfies \eqref{Fletcher}-\eqref{Goldstein} with $\sigma\in(0,\frac 1 2]$ for all $k>1$ ($\g_{k}\neq 0$), and if 
$$\sigma^{-\iota}\prod_{j=1}^{\infty} \left(\sigma \beta_{j}\right) <\infty,$$
 for a positive number $\iota$, $k>>\iota>1$ large enough such that $\lim_{i\rightarrow\iota}\sigma^{\iota}=0$; Then the NCG \eqref{NCG_update}-\eqref{NCG_direction} converges strongly in the following sense
\begin{equation}
\|\g_{k}\|_{2} \leq \|\g_{1}\|_{2}^{2}\left(\dfrac{1-\sigma}{1-2\sigma}\right) \sigma^{\iota} \left(\prod_{j=1}^{k-1} (\sigma\beta_{j}) \right)\dfrac{1}{\sigma^{\iota}}
\end{equation}
\end{theorem}

\begin{proof}
Is straightforward by bootstrapping argument. Indeed, starting from 
$$
{\left|\g_{k}^T\d_{k} \right|} \leq \beta_{k}\sigma {\left|\g_{k-1}^T\d_{k-1} \right|},
$$
leads to 
$$
{\left|\g_{k}^T\d_{k} \right|} \leq \sigma^{k-1}\prod_{j=1}^{k-1}\beta_{j} {\left|\g_{1}^T\d_{1} \right|},
$$
using \eqref{NCG_direction}$_1$ yields 
$$
{\left|\g_{k}^T\d_{k} \right|} \leq \sigma^{k-1}\prod_{j=1}^{k-1}\beta_{j} \| \g_{1} \|_{2}^{2}.
$$
The case of $\sigma\in(0,\frac 12)$ is straightforward, while for the extreme case of $\sigma=1/2$ we have the indeterminate form ($\infty\cdot0$) coming from
$(\dfrac{1-\sigma}{1-2\sigma}) \sigma^{\iota}$ which could easily be proven to be convergent to zero using l'Hospital's rule limit theorem as $\sigma\rightarrow 1/2$ with a relatively big enough $k>>\iota\geq 1$. Indeed, 
$$
\lim_{\sigma\rightarrow1/2} (\dfrac{1-\sigma}{1-2\sigma}) \sigma^{\iota} =
\lim_{\sigma\rightarrow1/2} \dfrac{\sigma^{\iota}}{(\dfrac{1-2\sigma}{1-\sigma})}\equiv
\lim_{\sigma\rightarrow1/2} \iota\sigma^{\iota-1} (1-\sigma)^2=0.
$$
The proof is complete.$\hfill\qed$
\end{proof}

It is worth noticing that, the condition on $\beta_{k}$ plays an important role on ensuring global convergence~\cite{Dai2011}. 

In the sequel we shall consider such $\beta=\max_{j}\beta_{j}$ satisfying $\sigma\beta\leq 1$. Furthermore, under boundedness condition of the eigenvalues of the second derivative of the objective function we show a linear convergence rate of the NCG. The convergence results is stated in the following theorem.

\begin{theorem}
The NCG method \eqref{NCG_update}-\eqref{NCG_direction} with $\beta_{k}$ as in \eqref{NCGMFR},\eqref{NCGMPR},\eqref{NCGMHS} and \eqref{NCGMDY} converges linearly with the following rate
\begin{equation}\label{NCG-LinearConvergence}
0\leq \left(1-\dfrac{1-\beta\sigma}{2-\beta\sigma}\dfrac{\alpha_\text{min}\lambda_\text{min}(1-\sigma)}{\alpha_\text{max}\lambda_\text{max}(1+\sigma)}\rho\right) \dfrac{1-\sigma}{1-2\sigma} < 1
\end{equation}
If
\begin{itemize}
\item the step-length $\alpha_{k}>0$ satisfies \eqref{Goldstein} with $\sigma\in(1,\frac 1 2]$ for all $k$, then the descent property of the Nonlinear conjugate gradient holds
\item the second derivative $\nabla^{2}\J$ is positive definite with extreme eigenvalues $\lambda_\text{max}$ and $\lambda_\text{min}$ respectively.
\item there exist an upper bound $\beta$ for $\beta_{k}$ such that $\beta\sigma\leq 1$
\end{itemize}
\end{theorem}
\begin{proof}
Assume $\lambda_\text{min}$ is the lowest eigenvalue of the second derivative positive definite operator $\nabla^{2}\J$. We have
$$
\int_{0}^{1}\alpha_{k}\d_{k}^T \nabla^{2}\J(\u_{k}+s\alpha_{k}\d_{k})\d_{k}\, ds  = \g_{k+1}^{T}\d_{k} -  \g_{k}^{T}\d_{k}
$$
which gives 
$$
\alpha_{k} \lambda_\text{min} \|\d_{k}\|^{2} \leq -(\sigma+1) \g_{k}^{T}\d_{k}
$$
having used \eqref{Goldstein}. Therefore we have the following upper bound of the magnitude of the directions $\d_{k}$
\begin{equation}\label{UpperBoundMagnitudeofd}
\|\d_{k}\|^{2} \leq \dfrac{\sigma+1}{\lambda_\text{min}\alpha_{k}} |\g_{k}^{T}\d_{k}|.
\end{equation}
Furthermore, If $\lambda_\text{max}$ is an upper bound to $\|\nabla^2\J(\u)\|$, we have,
\begin{equation*}
\g_{k+1}^T\d_{k}\leq \g_{k}^T\d_{k} + \lambda_\text{max} \alpha_{k}\|\d_{k}\|^{2}
\end{equation*}
which we combine with \eqref{Goldstein} to obtain
\begin{equation*}
\alpha_{k} \geq -\dfrac{1-\sigma}{\lambda_\text{max}\|\d_{k}\|_{2}^{2}} \g_{k}^T\d_{k} 
\end{equation*}
which in its turn plugged into \eqref{Fletcher} yields
\begin{equation*}
\tilde\J_{k}-\tilde\J_{k+1} \geq \dfrac{\rho(1-\sigma)}{\lambda_\text{max}} \dfrac{|\g_{k}^T\d_{k}|^2}{\|\d_{k}\|^{2}}
\end{equation*}
Further use of \eqref{UpperBoundMagnitudeofd} gives
\begin{equation*}
\tilde\J_{k}-\tilde\J_{k+1} \geq \dfrac{\lambda_\text{min}(1-\sigma)}{\lambda_\text{max}(1+\sigma)}\rho\alpha_{k} |\g_{k}^T\d_{k}|
\end{equation*}
which we sum up to obtain 
\begin{equation}\label{lowerBoundForJ}
\tilde\J_{k} \geq \dfrac{\lambda_\text{min}(1-\sigma)}{\lambda_\text{max}(1+\sigma)}\rho\alpha_\text{min} \sum_{\ell=k}^{\infty} |\g_{k}^T\d_{k}|
\end{equation}

In the other hand, if we consider the real valued function $\theta(s)=-\alpha_{k}\g(\u_{k}+s\alpha_{k}\d_{k})^{T}\d_{k}$ for $s\in(0,1)$. We have $\theta^\prime \leq 0$ hence it is decreasing function over $(0,1)$. 

	Indeed, $\theta^\prime(s) = -\alpha_{k}^2 \d_{k}^{T}\nabla^{2}\J(\u_{k}+s\alpha_{k}\d_{k}) \d_{k}$ is negative since $\nabla^2\J$ is positive matrix. It follows immediately the following upper bound that overestimates the integral 
$$\int_{0}^{1} \theta(s)\,ds \leq \theta(0)$$
which rewrites 
$$
-\int_{0}^{1}\alpha_{k}\g(\u_{k}+s\d_{k})^{T}\d_{k} \, ds \leq -\alpha_{k}\g_{k}^{T}\d_{k} \leq \alpha_\text{max} |\g_{k}^{T}\d_{k}|
$$
 Thus, 
\begin{equation*}\label{}
\tilde\J_{k}-\tilde\J_{k+1} \leq \alpha_\text{max} |\g_{k}^{T}\d_{k}|
\end{equation*}
Henceforth, summing up term from $k$ to infinity we have 
\begin{equation*}\label{}
\tilde\J_{k} \leq \alpha_\text{max} \sum_{\ell=k}^{\infty} |\g_{\ell}^{T}\d_{\ell}|
\end{equation*}
which we rewrite as 
\begin{eqnarray}\label{}
\tilde\J_{k} &\leq& \alpha_\text{max} \left|\g_{k}^T\d_{k}\right| \left(1+ \sum_{\ell\geq k+1} \dfrac{\left|\g_{\ell}^T\d_{\ell}\right|}{\left|\g_{k}^T\d_{k}\right|} \right).
\end{eqnarray}
Now, we use the results of Lemma~\ref{LemmaRIAHI} to state that 
$$\sum_{\ell\geq k+1} \dfrac{\left|\g_{\ell}^T\d_{\ell}\right|}{\left|\g_{k}^T\d_{k}\right|} \leq  \sum_{\ell\geq k+1} (\beta\sigma)^{\ell-k}= \sum_{\ell\geq0} (\beta\sigma)^\ell \leq \sum_{\ell\geq0} (\beta\sigma)^\ell = \dfrac{1}{1-\beta\sigma}.$$
and obtain, 
\begin{equation}
\tilde\J_{k} \leq \alpha_\text{max}\dfrac{2-\beta\sigma}{1-\beta\sigma} \left|\g_{k}^T\d_{k}\right|\label{upperBoundForJ},
\end{equation}
 
 	 Henceforth, Combining \eqref{upperBoundForJ} with \eqref{lowerBoundForJ} we obtain 
\begin{equation*}
\dfrac{1-\beta\sigma}{2-\beta\sigma}\dfrac{\alpha_\text{min}\lambda_\text{min}(1-\sigma)}{\alpha_\text{max}\lambda_\text{max}(1+\sigma)}\rho \sum_{\ell=k}^{\infty} |\g_{\ell}^T\d_{\ell}|\leq \left|\g_{k}^T\d_{k}\right|
\end{equation*}
Let $\kappa:=\dfrac{1-\beta\sigma}{2-\beta\sigma}\dfrac{\alpha_\text{min}\lambda_\text{min}(1-\sigma)}{\alpha_\text{max}\lambda_\text{max}(1+\sigma)}\rho $. We have $0\leq \kappa\leq 1$.

	 Now consider $\mathcal R_{k} = \sum_{\ell\geq k} |\g_{k}^{T}\d_{k}|$ and rewrite the above inequality to 
$$\kappa\mathcal R_{k}  \leq \left(\mathcal R_{k} - \mathcal R_{k+1}\right) .$$
Hence, 
$$\mathcal R_{k+1}  \leq  \left(1 -\kappa\right) \mathcal R_{k}  .$$
The result follows from \eqref{AL-BAALI-INEQ}.  
The proof is complete.$\hfill\qed$
\end{proof}

\section{Numerical illustration}\label{NumShow}
We present in this section, some applications of parameter estimation gradient-based methods to certain nonlinear models. We shall present the optimality condition conducted with the optimal control problem, as described earlier in Section \ref{sec:GS}. 
	
	In order to avoid falling into local minima, a sampling procedure has been utilized with a small step (depending on the handled problem). This procedure generates a coarse grid over which a global search is made, which helps to pre-estimates the initial guess for the optimization algorithm. This is a popular approach in parameter estimations problems that we will use in all our numerical experiments.
	
	In the sequel, all optimization procedures are concerned with the minimization of the following objective functional 
$$\J(\u) = \dfrac{\alpha}{2}\|\u(t)\|_{2}^{2} + \dfrac{1}{2}\int_{t_i}^{t_f}\|\y(t,\u)-\y^{T}(t)\|_{\V}^{2} \dt$$ 
where the state variable is subject to the constraint of the given dynamical model.
\subsection{FitzHugh-Nagumo model}
In the FitzHugh-Nagumo (FHN) model,
\begin{equation}\label{FHN}
\begin{array}{ccc}
\dot{v}(t)   &=& -v(v-a)(v-1)+I_0\\
\dot{w}(t) &=& \epsilon (v -dw).
\end{array}\end{equation}
 $v$ stands for the membrane potential measurable variable, $w$ stands for the unmeasurable recovery variable. $I_0$ stands for the injected current stimulus, while $a,\epsilon,d$ represent the unknown parameters. Following Section~\ref{sec:GS} we have

\begin{eqnarray*}
 \partial_{\y}\F(t,\y,\gamma)
 \begin{pmatrix}
 \delta v \\ \delta w
 \end{pmatrix}&=&-
  \begin{pmatrix}
(\gamma_{1}+3\gamma_{2} v^2)\delta v - \delta w \\
\gamma_{4}\delta v + \gamma_{3}\delta w
 \end{pmatrix}\\
\partial_{\gamma}\F(t,\y,\gamma)
\begin{pmatrix}\delta\gamma_{1}\\\delta\gamma_{2}\\\delta\gamma_{3}\\\delta\gamma_{4}\\\delta\gamma_{5}\end{pmatrix}
&=&
\begin{pmatrix}
\delta\gamma_{1} v + \delta\gamma_{2} v^3\\\delta\gamma_{4} v+\delta\gamma_{3} w + \delta\gamma_{5}
\end{pmatrix}.
\end{eqnarray*}
The optimality system for the FHN model includes \eqref{FHN} with the adjoint state equation 
\begin{equation}\begin{array}{ccc}
 \dot{p}(t) &=& -\gamma_{1} p(t) - 3\gamma_{2}v^3(t) p(t) + \gamma_{4} q(t) +v(t)-v^T(t)\\
\dot{q}(t) &=& p(t)+\gamma_3 q(t) +w(t)-w^T(t).
\end{array}\end{equation}
Thus the gradient writes
\begin{equation}
\g(\gamma) = \alpha\begin{pmatrix}\gamma_{1}\\\gamma_{2}\\\gamma_{3}\\\gamma_{4}\\\gamma_{5}\end{pmatrix}+
\begin{pmatrix}
v p \\ v^2 p \\ w q \\ v q \\ q
\end{pmatrix}.
\end{equation}

The numerical simulation of the parameter estimation for the FHN model are reported in Table~\ref{FHN-data} supplemented with plots of the fit that are depicted in Fig.\ref{FHN-plot}. The results show identification of the parameters $a,b,\epsilon$ and $d$ involved in the mathematical model \eqref{FHN}.

\begin{figure}[htbp]
\begin{tabular}{ccc}
\raisebox{-.6\height}{\includegraphics[height=5cm,width=6cm]{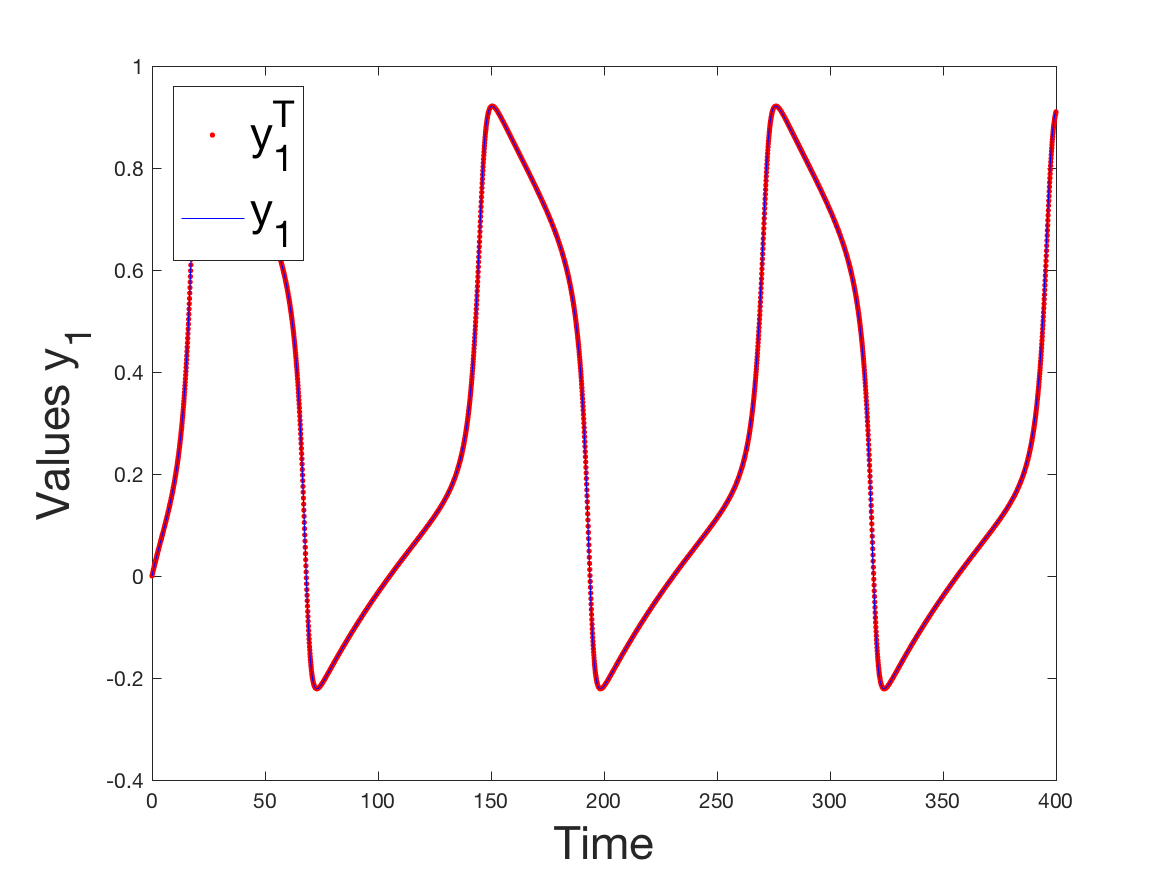}}&
\raisebox{-.6\height}{\includegraphics[height=5cm,width=6cm]{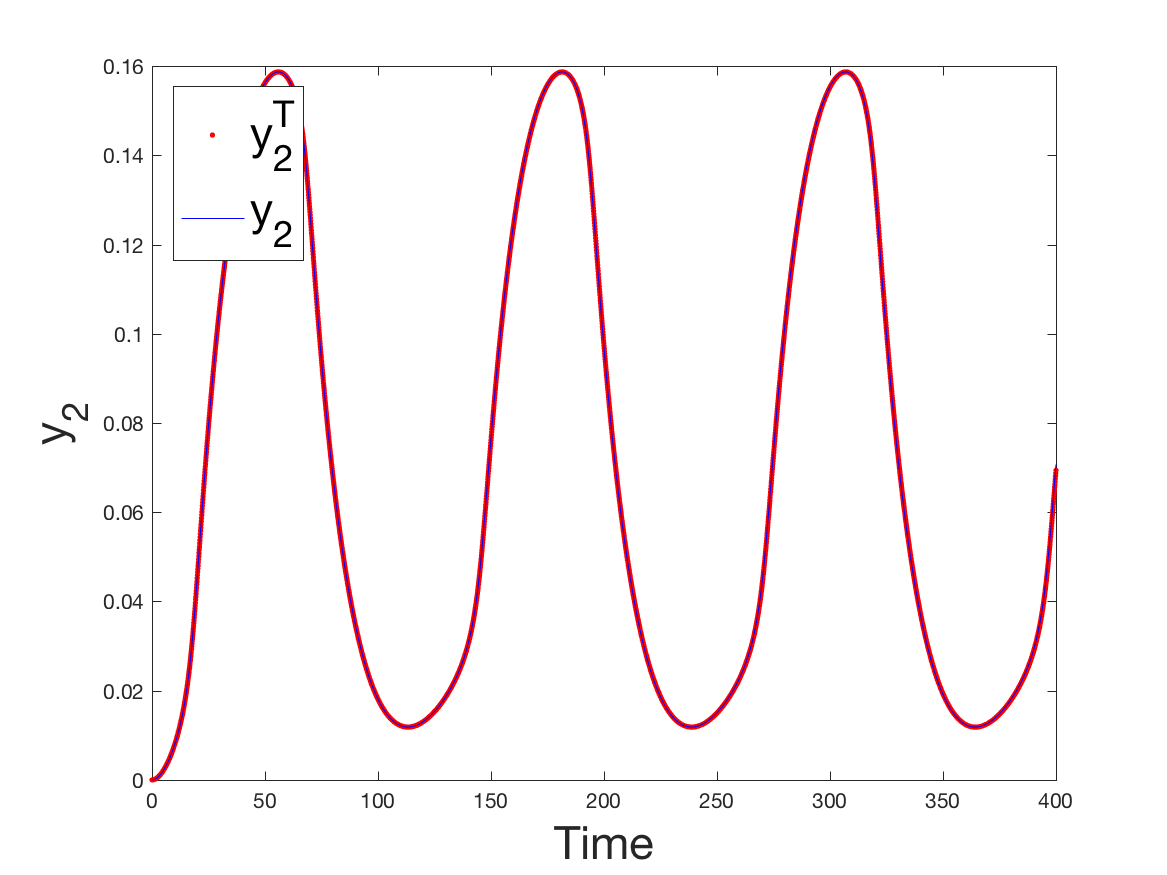}}&\rotatebox{-90}{Noise Free}\\
\raisebox{-.6\height}{\includegraphics[height=5cm,width=6cm]{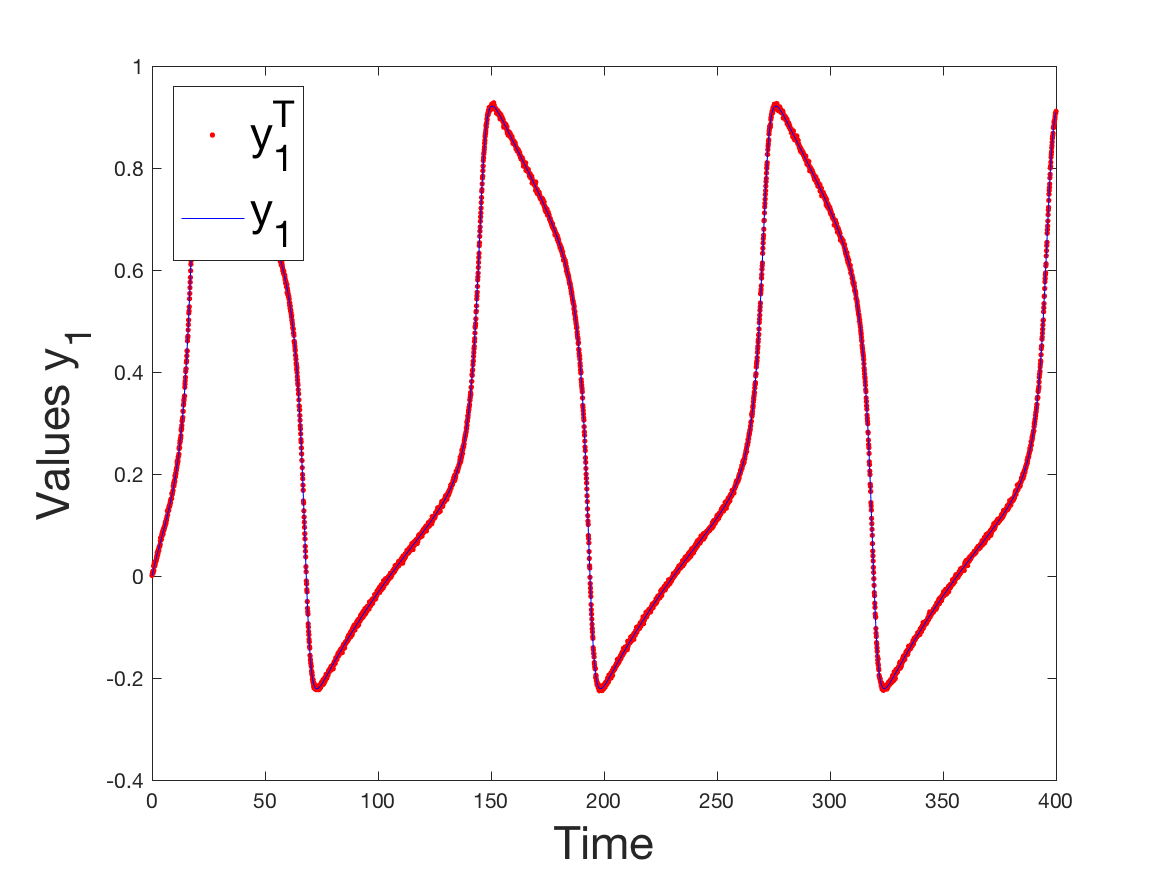}}&
\raisebox{-.6\height}{\includegraphics[height=5cm,width=6cm]{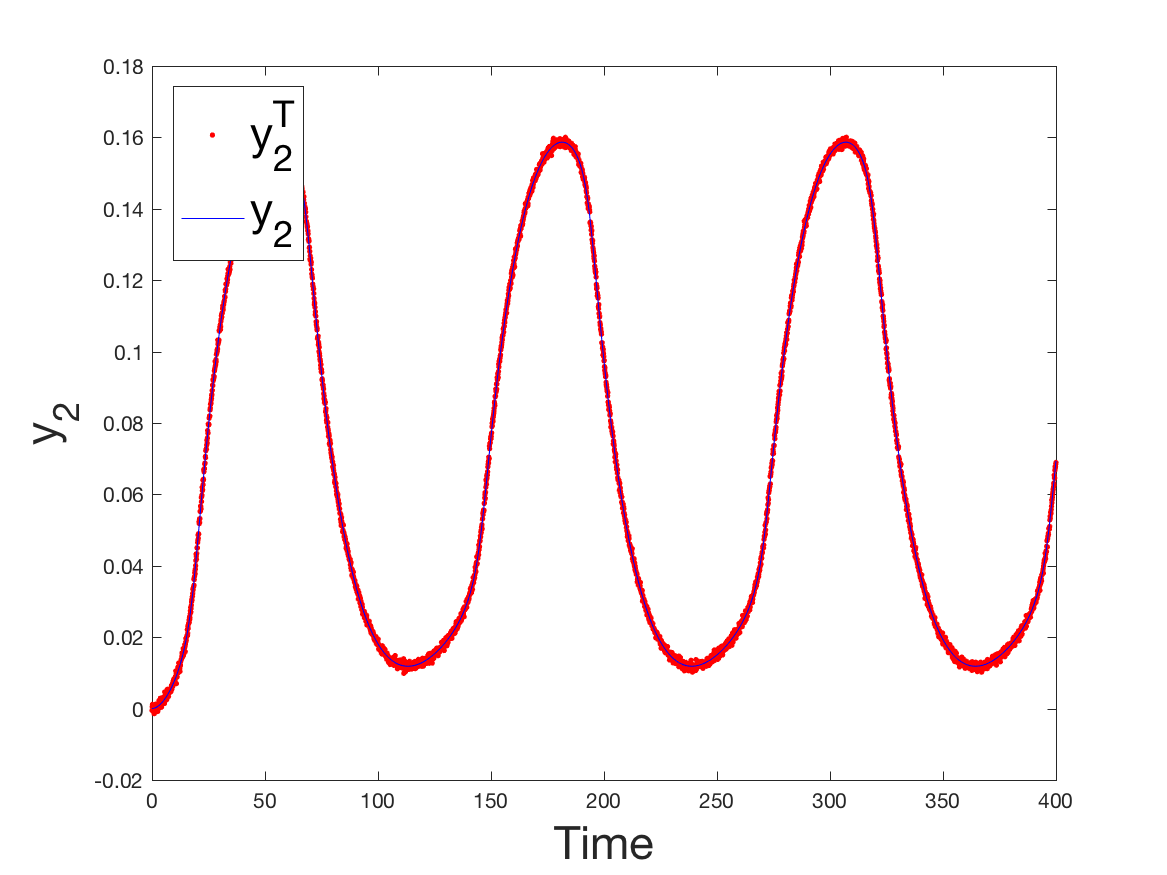}}&\rotatebox{-90}{1$\%$Noise}\\
\raisebox{-.6\height}{\includegraphics[height=5cm,width=6cm]{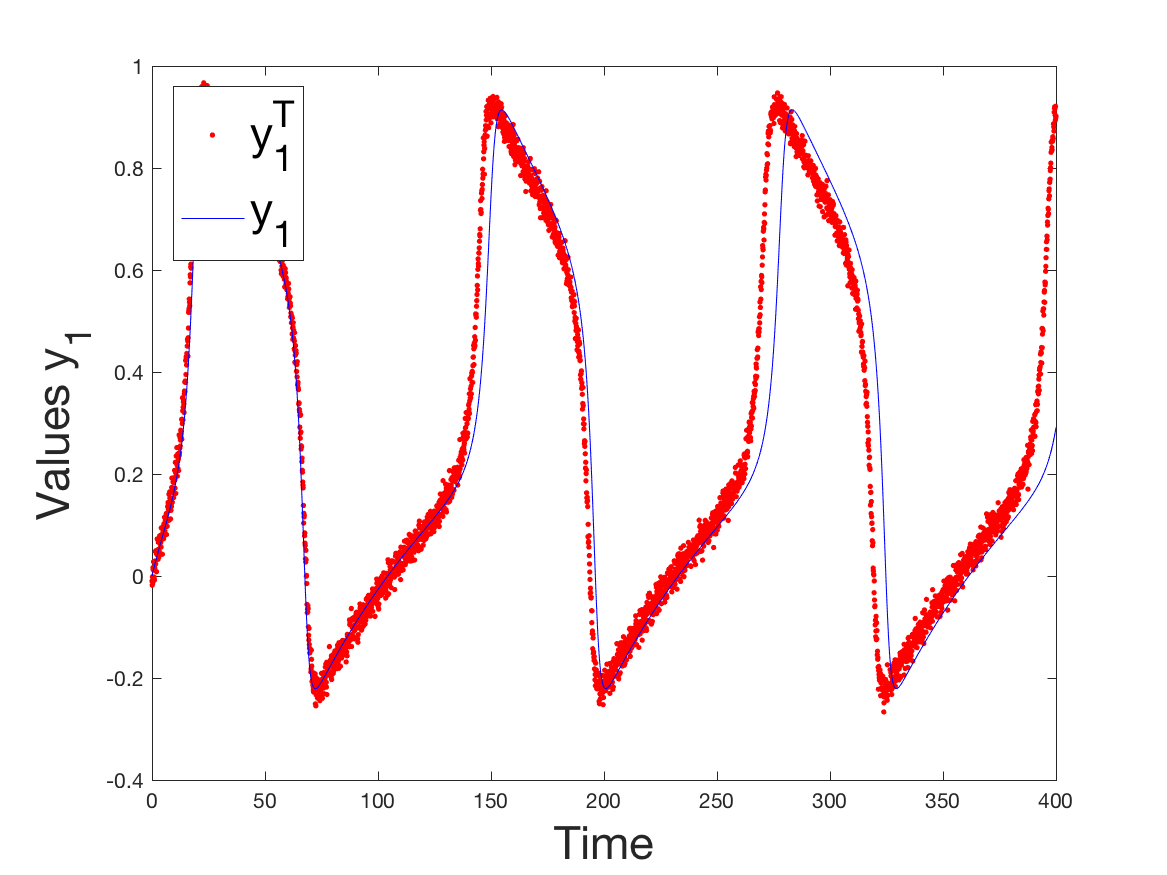}}&
\raisebox{-.6\height}{\includegraphics[height=5cm,width=6cm]{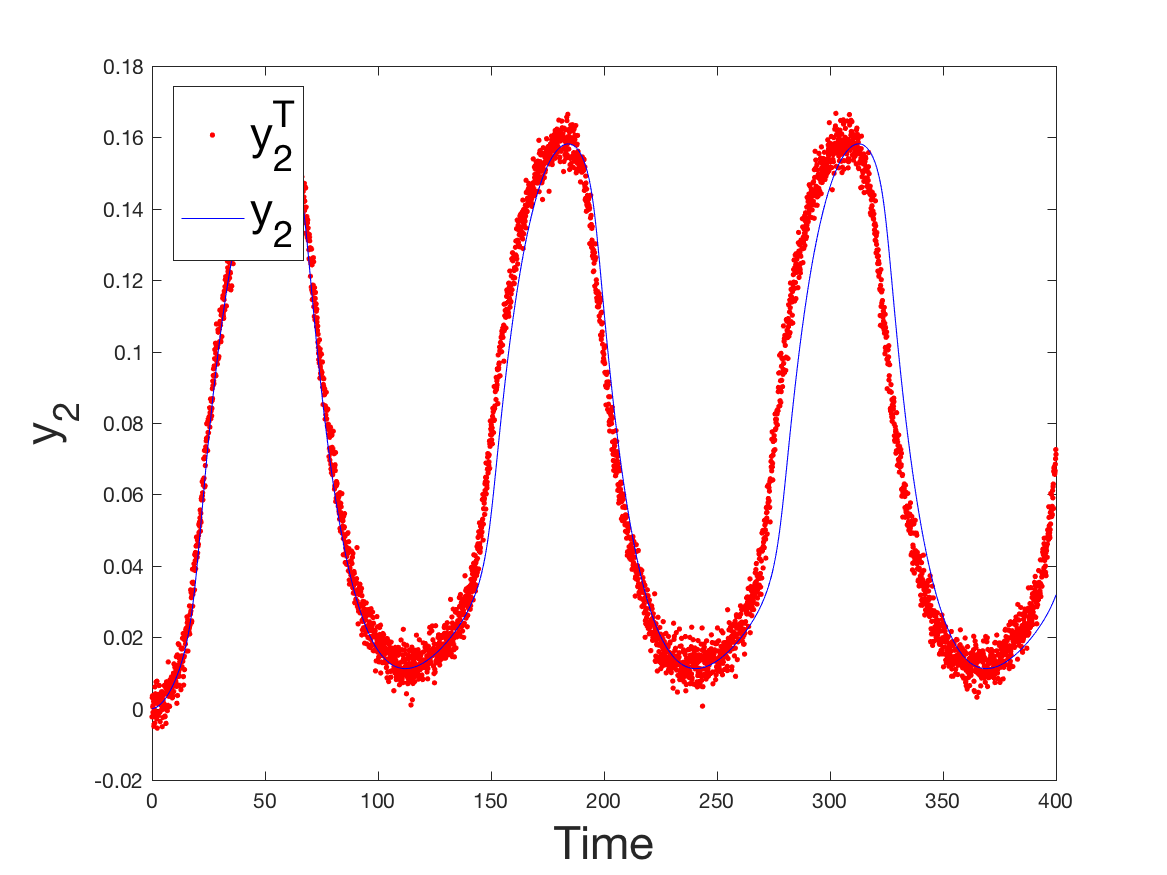}}&\rotatebox{-90}{5$\%$Noise}\\
\raisebox{-.6\height}{\includegraphics[height=5cm,width=6cm]{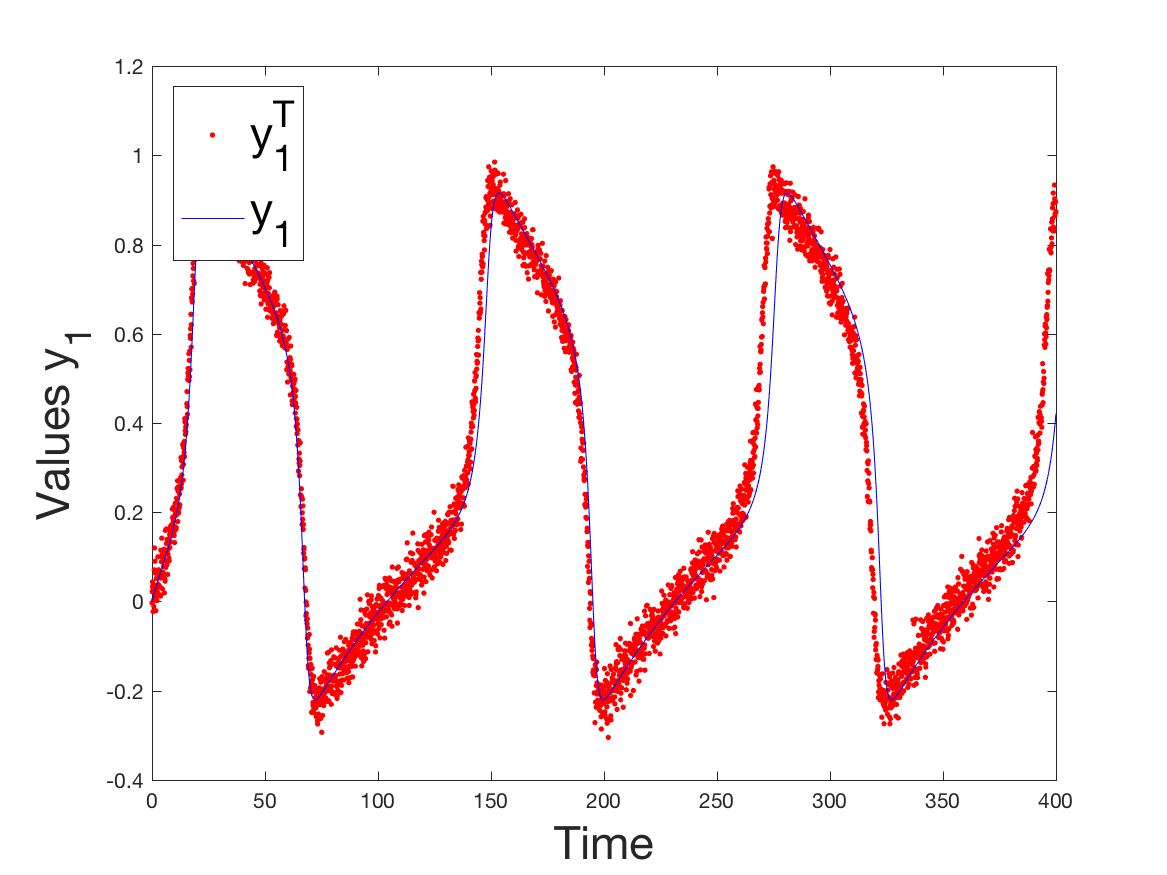}}&
\raisebox{-.6\height}{\includegraphics[height=5cm,width=6cm]{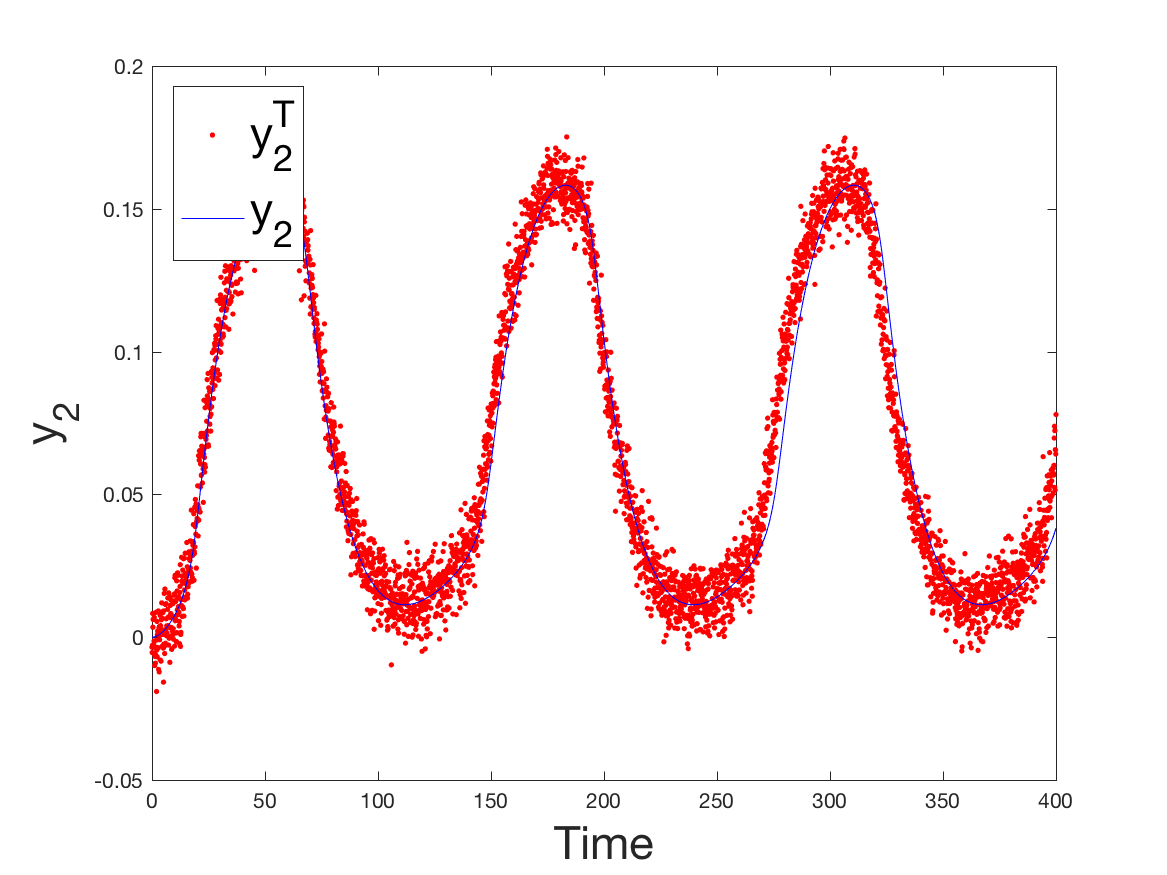}}&\rotatebox{-90}{10$\%$Noise}\\
\end{tabular}
\caption{Numerical experiments of the FitzHug-Nagumo Model fit with a noisy data (0\%,1\%,5\% and 10\% from top to bottom). Profile of the solution $\y_1$ (left) and $\y_2$ (right) with their respective targeted profile $\y_1^T$ and $\y_2^T$ respectively.}\label{FHN-plot}
\end{figure}

\begin{table}[!htpb]\begin{center}
\begin{tabular}{c|c|c|c|c}
\hline \text{Coef Variables} & $a$ &  $I_0$ &   $\epsilon$ & $d$\\ \hline
          \text{Actual Val.} & 0.1 & 2.0e-2 & 1.0e-02 & 4.0 \\
\text{Noise-free Val.} &         1.091342e-01 &     2.0522192e-02 &     1.035301e-02 &     4.007462e+00\\
1\% \text{Noise} 
&     1.00103e-01
&     2.00172e-02
&     9.99849e-03
&     4.00009e+00\\
5\% \text{Noise} 
     & 1.003527e-01
     & 1.912007e-02
     &1.023820e-02
     &4.000431e+00\\
10\% \text{Noise}  
&   1.009941e-01
&   1.937421e-02
&   1.023523e-02
&   4.000754e+00\\\hline
\end{tabular}\caption{Convergence of the parameter estimation algorithm toward the solution of FitzHug-Nagumo}\label{FHN-data}
\end{center}\end{table}

\subsection{Dynamical system with Matrix parametrized coefficients}

We shall consider a simple Coupled system of non-homogenous ODEs with variable coefficients. These coefficients are considered to be derived from a given nonlinear model. The resulting system to be solved reads
\begin{eqnarray*}
\dot{y}_{1} &=& a_{11}(c_{1},c_{2},c_{3}) y_{1} + a_{12}(c_{1},c_{2},c_{3}) y_{2}+\sin(t)\\
\dot{y}_{2} &=& a_{21}(c_{1},c_{2},c_{3}) y_{1} + a_{22}(c_{1},c_{2},c_{3}) y_{2}
\end{eqnarray*}
Where $a_{11},a_{12},a_{21}$ and $a_{22}$ are possibly nonlinear functions of the variables $c_{1},c_{2}$ and $c_{3}$.
The above system writes simply in a vector form 
\begin{equation}\label{linear1}
\dot\y  = A(c)\y + \f(t)
\end{equation}
where $\y=\left(y_1,y_2\right)^T$,  $c=(c_1,c_2,c_3)$ and 
\begin{equation}
\label{ExampleConstraint}
A(c)=
\begin{pmatrix}
a_{11} & a_{12}\\a_{21} & a_{22}\\
\end{pmatrix}
\end{equation}
In \eqref{ExampleConstraint}, we consider the following ``artificial'' model 
\begin{equation}
\begin{cases}
a_{11}(c) = c_{1}^2c_{2}\\
a_{12}(c) = c_{2}c_{3}\\
a_{21}(c) = \sin(c_3)\\
a_{22}(c) = c_{1}c_{3}^2
\end{cases}
\end{equation}
It follows that 
\begin{equation}\label{linear2}
\begin{cases}
a_{11}(c+h) = a_{11}(c) + c_1^2 h_2 + 2c_1h_1h_2 + h_1^2c_2\\
a_{12}(c+h) = a_{12}(c) + c_2h_3+h_2c_3  + h_2h_3\\
a_{21}(c+h) =  a_{21}(c) + h_3\cos(c_3) - \dfrac {h^2_3}{2}\sin(c_3) + o(h^3_3)\\
a_{22}(c+h) = a_{22}(c) +c_3^2 h_1 + 2c_3h_3h_1 + h_3^2c_1.
\end{cases}
\end{equation}
Henceforth the following equation holds
\begin{equation}
A(c+h) = A(c)  + L(c) h + R(c;h)
\end{equation}
with 
\begin{equation}L(c;h)=
\begin{pmatrix}
c_1^2h_2 & c_2h_3+h_2c_3\\
\cos(c_3)h_3 & h_1c_3^2\\
\end{pmatrix}
\end{equation}
This linear operator is Lipschitz as we have 
$$
\|L(c;v) - L(c;w) \| \leq \|v-w\|
$$
Indeed, 
\begin{eqnarray*}
\|L(c;v) - L(c;w) \|  &=& \left\|
\begin{pmatrix}
c_1^2(v_2-w_2)   & c_2(v_3-w_3) + c_3(v_2-w_2)\\
\cos(c_3)(v_3-w_3) & c_3^2(v_1-w_1)
\end{pmatrix}\right\| \\
&=& c_1^2(v_2-w_2)^2 + 
c_2^2(v_3-w_3)^2 + c_3^2(v_2-w_2)^2 \\&&+ 2c_2c_3(v_3-w_3) (v_2-w_2)  \\&&+
\cos(c_3)^2(v_3-w_3)^2 + c_3^4(v_1-w_1)^2\\
&\leq& c_1^2(v_2-w_2)^2 + 
c_2^2(v_3-w_3)^2 + c_3^2(v_2-w_2)^2 \\&&+ c_2^2c_3^2(v_3-w_3)^2 + (v_2-w_2)^2  \\&&+
\cos(c_3)^2(v_3-w_3)^2 + c_3^4(v_1-w_1)^2\\
&\leq& \max\{c_1^2,c_2^2,c_3^2,c_3^4,\cos(c_3).c_2^2c_3^2,1\} \|v-w\|
\end{eqnarray*}
and 
\begin{equation}R(c;h)=
\begin{pmatrix}
2c_1h_1h_2 + h_1^2c_2 & h_2h_3\\
- \dfrac {h^2_3}{2}\sin(c_3) + o(h^3_3) & 2c_3h_3h_1 + h_3^2c_1.\\
\end{pmatrix}
\end{equation}

$$\g(c) = \alpha c + \ell(c)$$
with 
$$\ell(c)=\begin{pmatrix}
c_3^2\int_{t_i}^{t_f}y_2p_2\dt\\
c_1^2\int_{t_i}^{t_f}y_1p_1+c_3\int_{t_i}^{t_f}y_2p_1\dt\\
c_2\int_{t_i}^{t_f}y_2p_1\dt+\cos(c_3)\int_{t_i}^{t_f}y_1p_1\dt
\end{pmatrix}$$

\begin{figure}[htbp]
\begin{tabular}{ccc}
\raisebox{-.6\height}{\includegraphics[height=5cm,width=6cm]{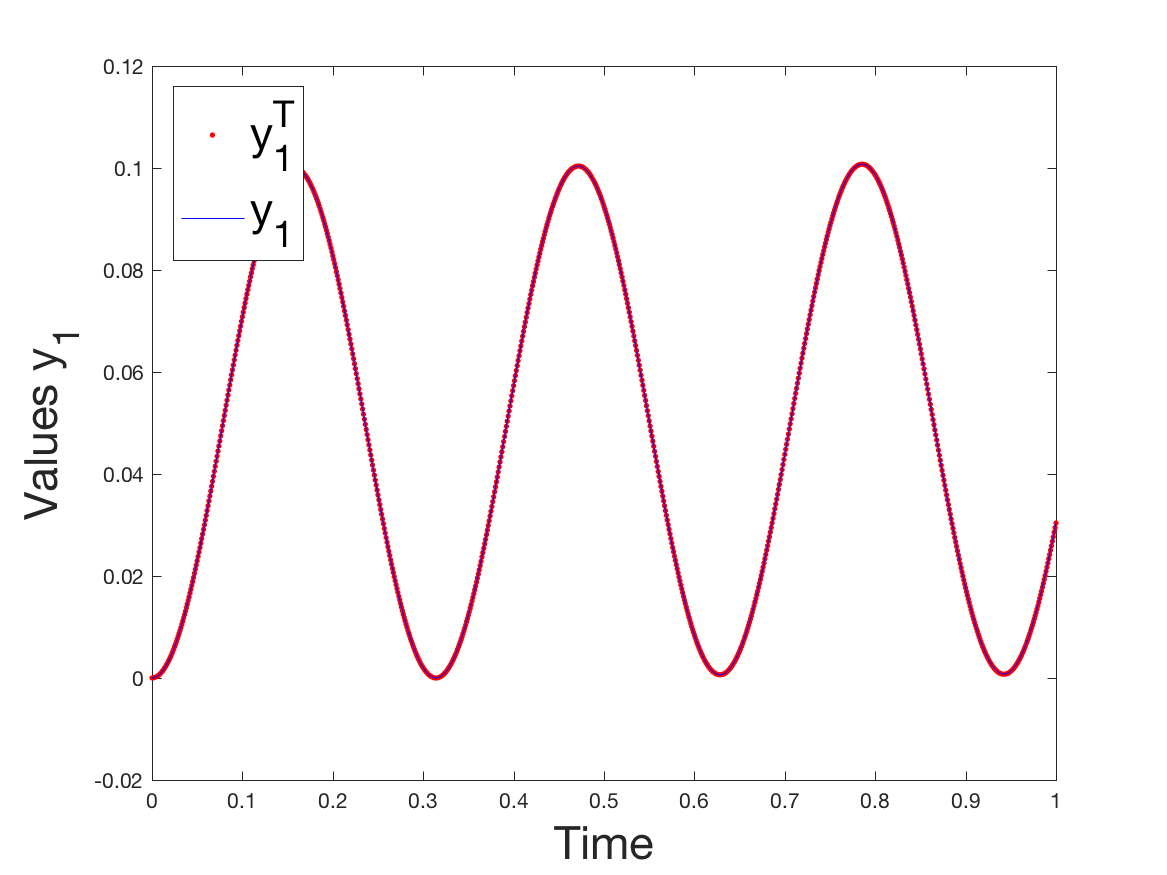}}&
\raisebox{-.6\height}{\includegraphics[height=5cm,width=6cm]{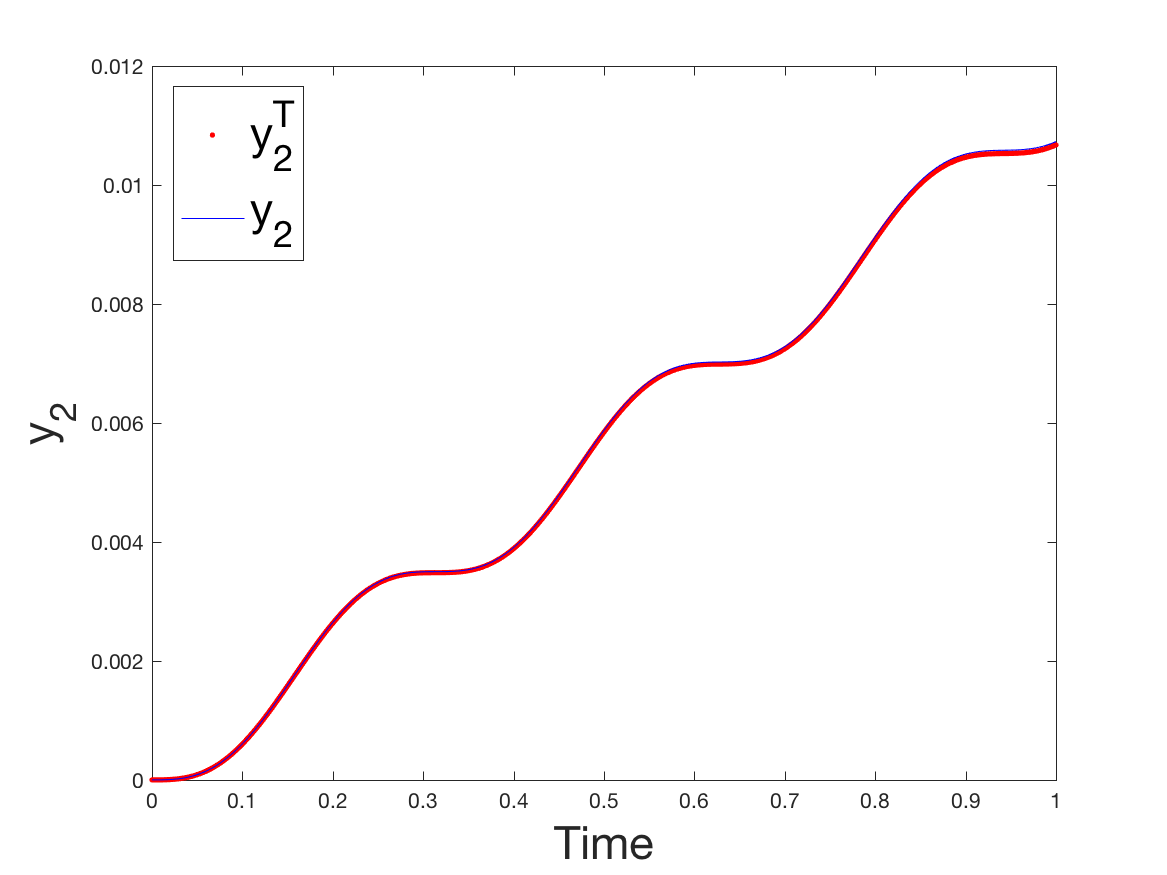}}&\rotatebox{-90}{Noise Free}\\
\raisebox{-.6\height}{\includegraphics[height=5cm,width=6cm]{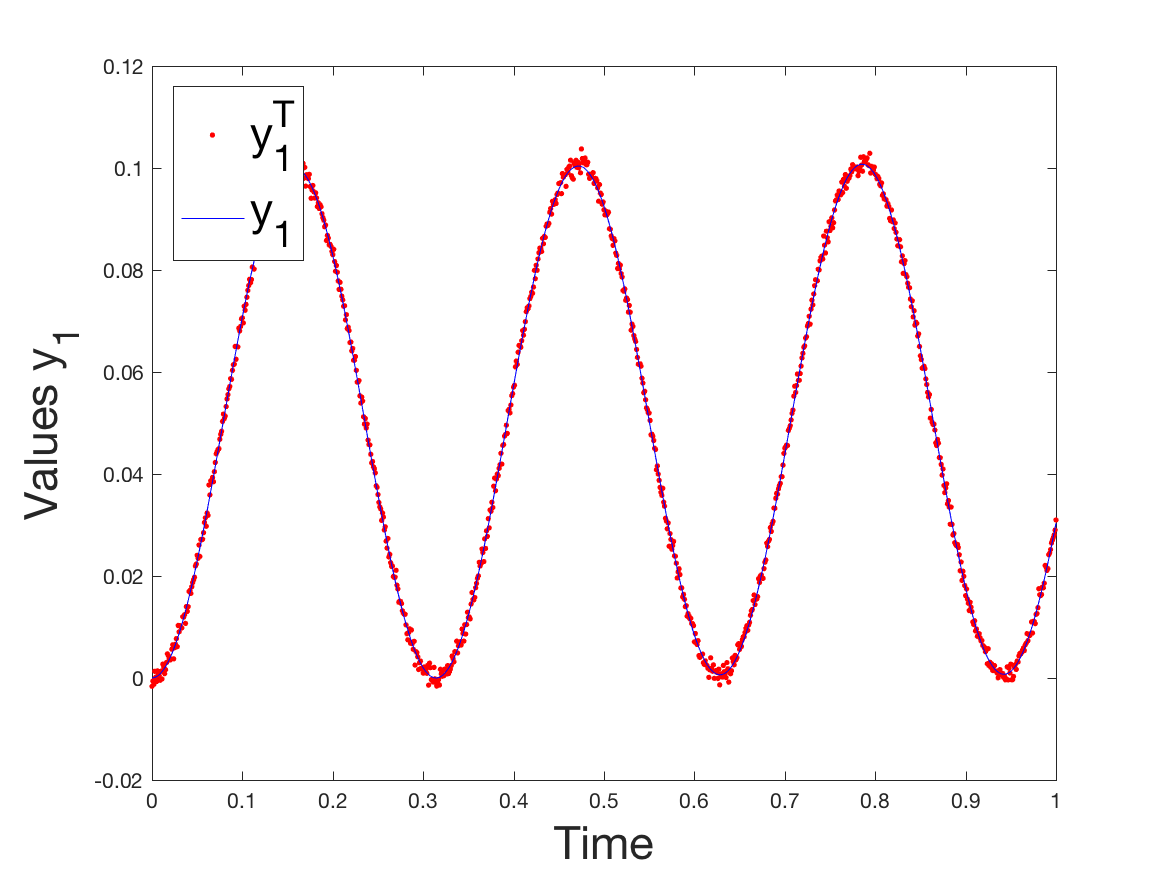}}&
\raisebox{-.6\height}{\includegraphics[height=5cm,width=6cm]{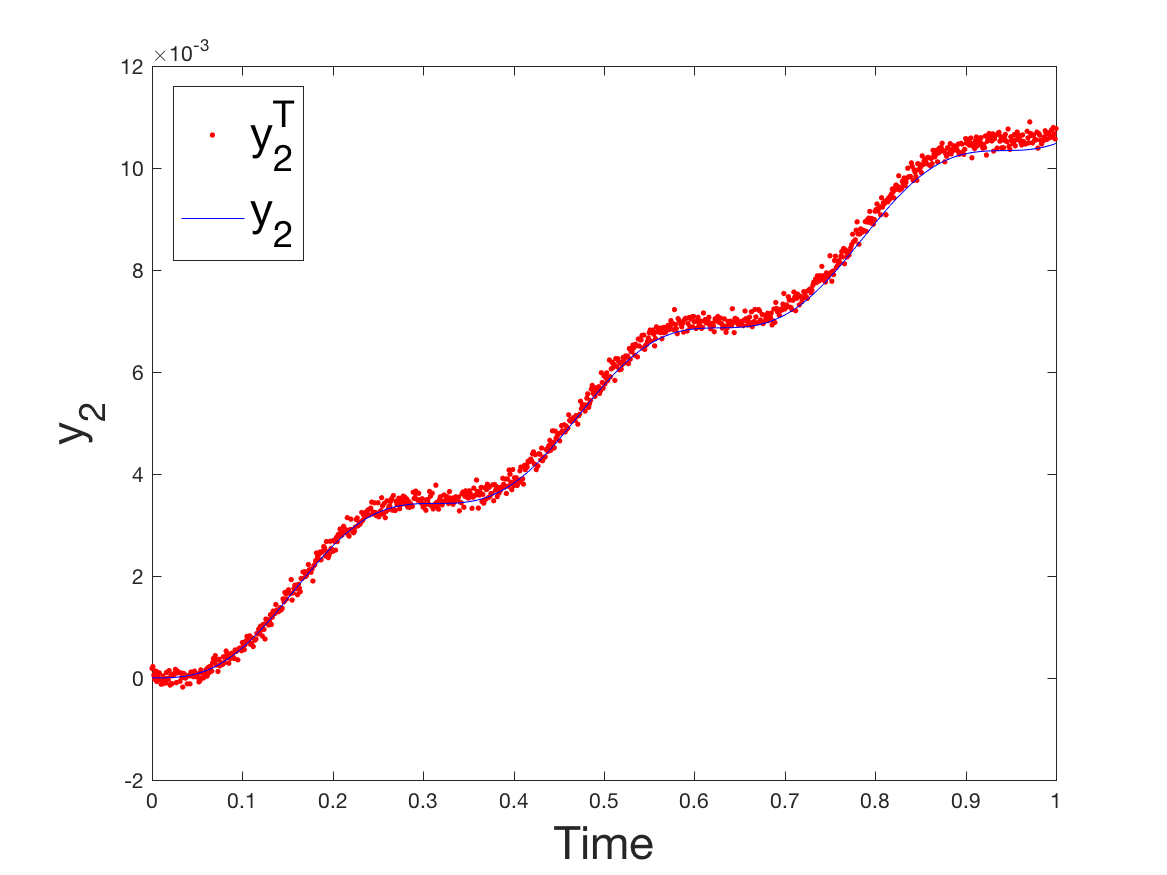}}&\rotatebox{-90}{1$\%$Noise}\\
\raisebox{-.6\height}{\includegraphics[height=5cm,width=6cm]{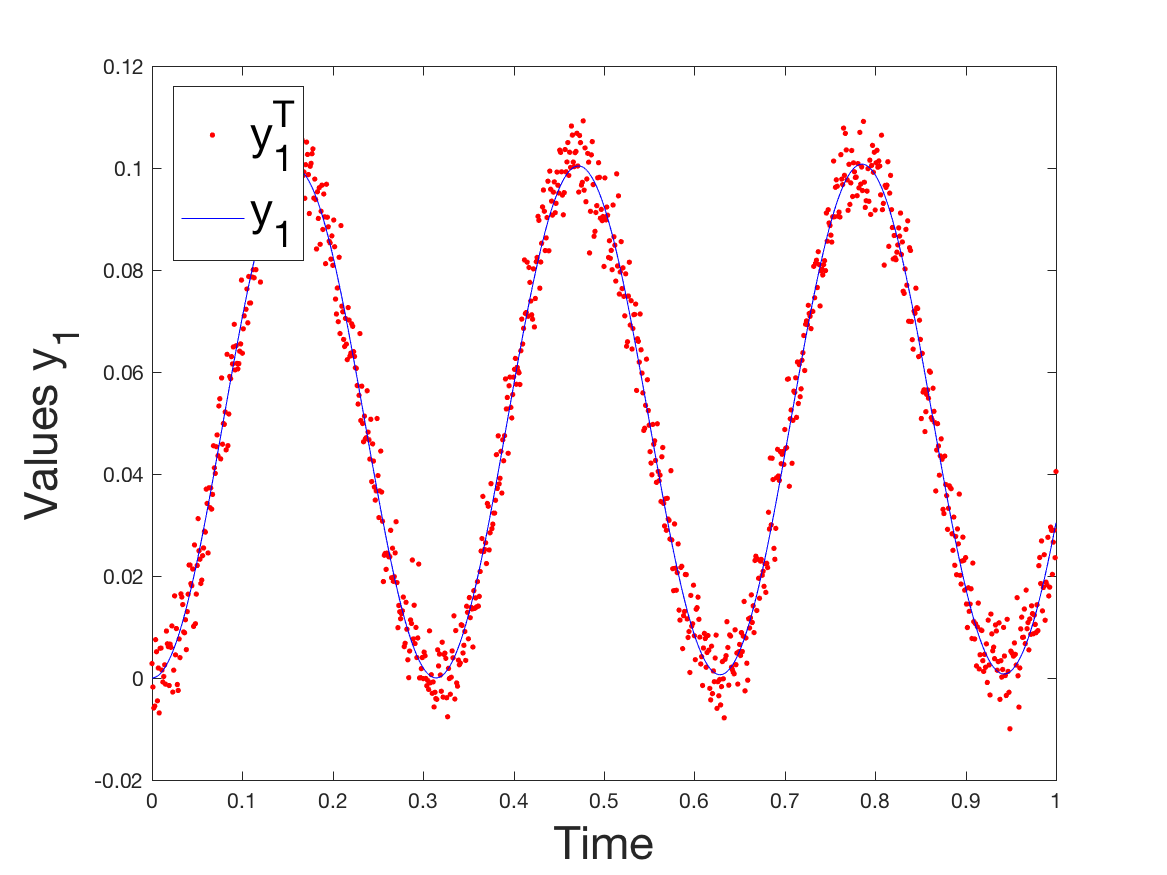}}&
\raisebox{-.6\height}{\includegraphics[height=5cm,width=6cm]{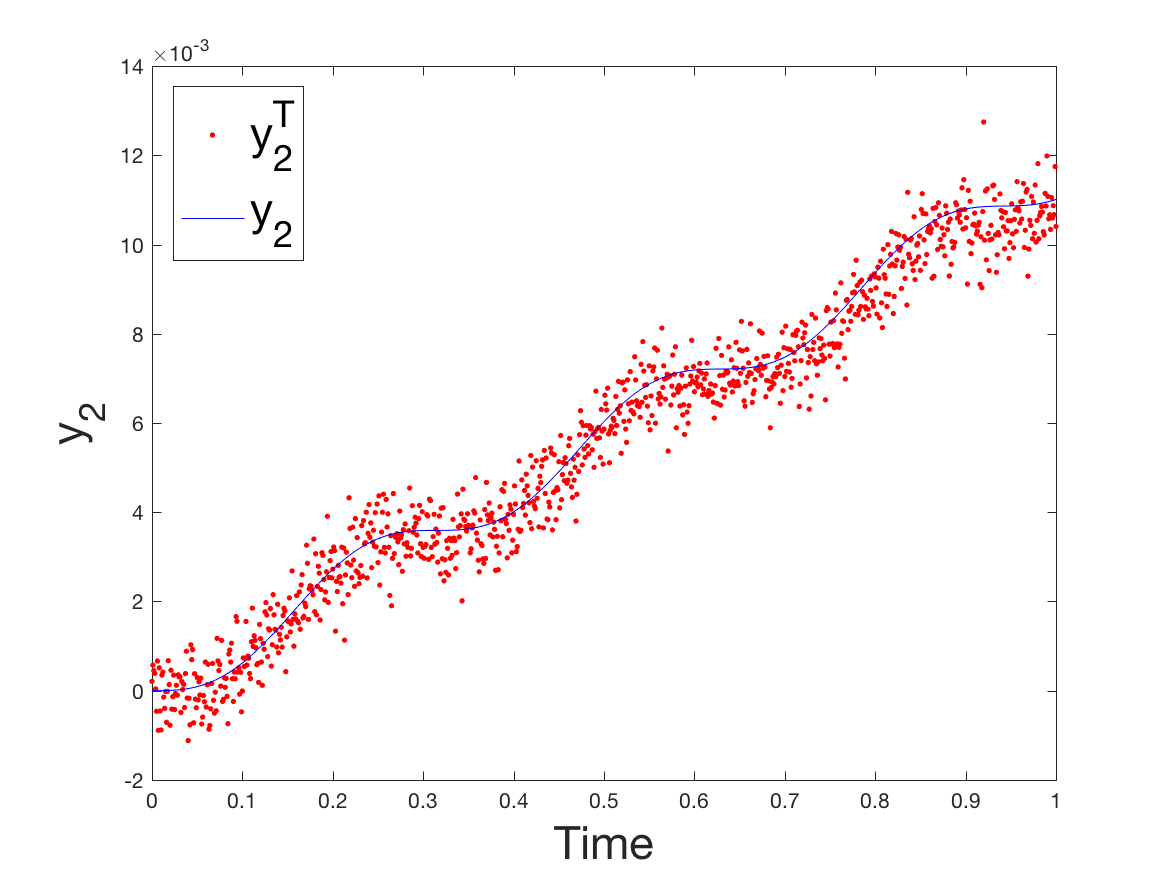}}&\rotatebox{-90}{5$\%$Noise}\\
\raisebox{-.6\height}{\includegraphics[height=5cm,width=6cm]{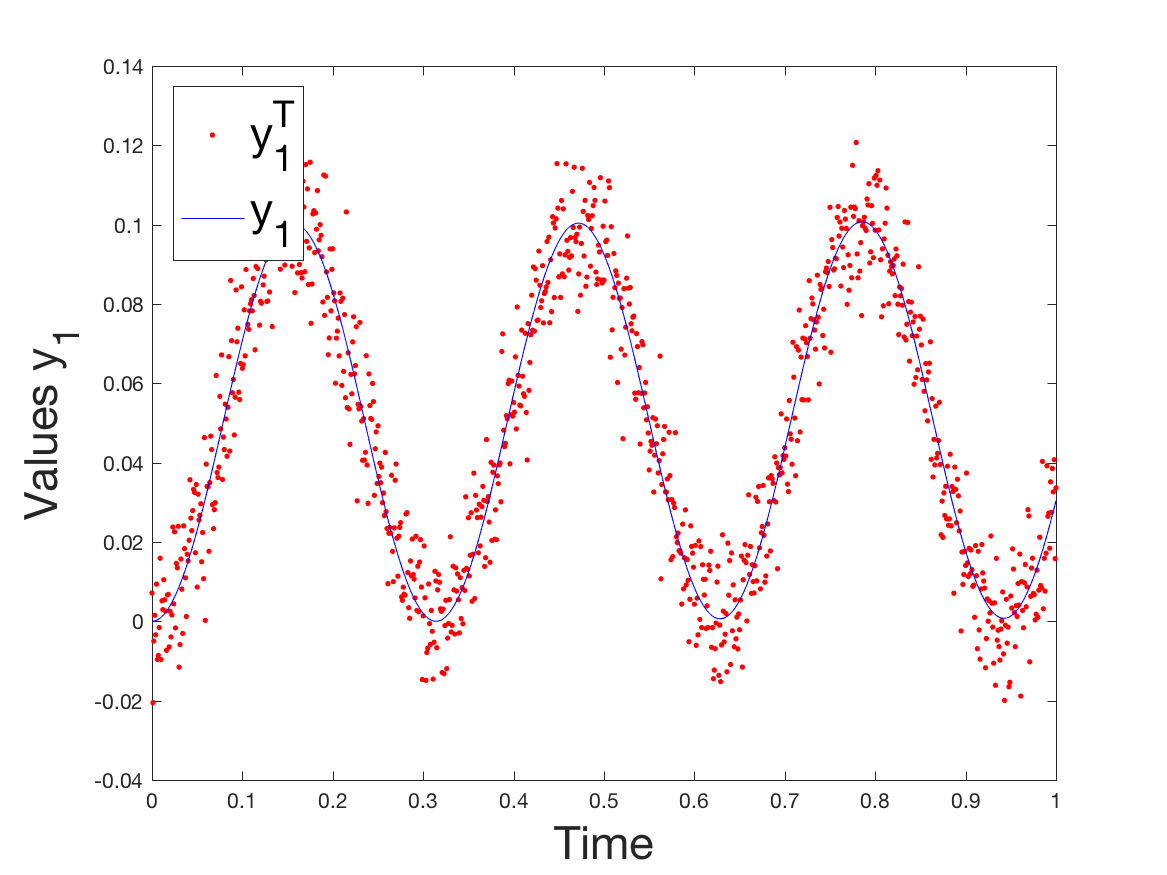}}&
\raisebox{-.6\height}{\includegraphics[height=5cm,width=6cm]{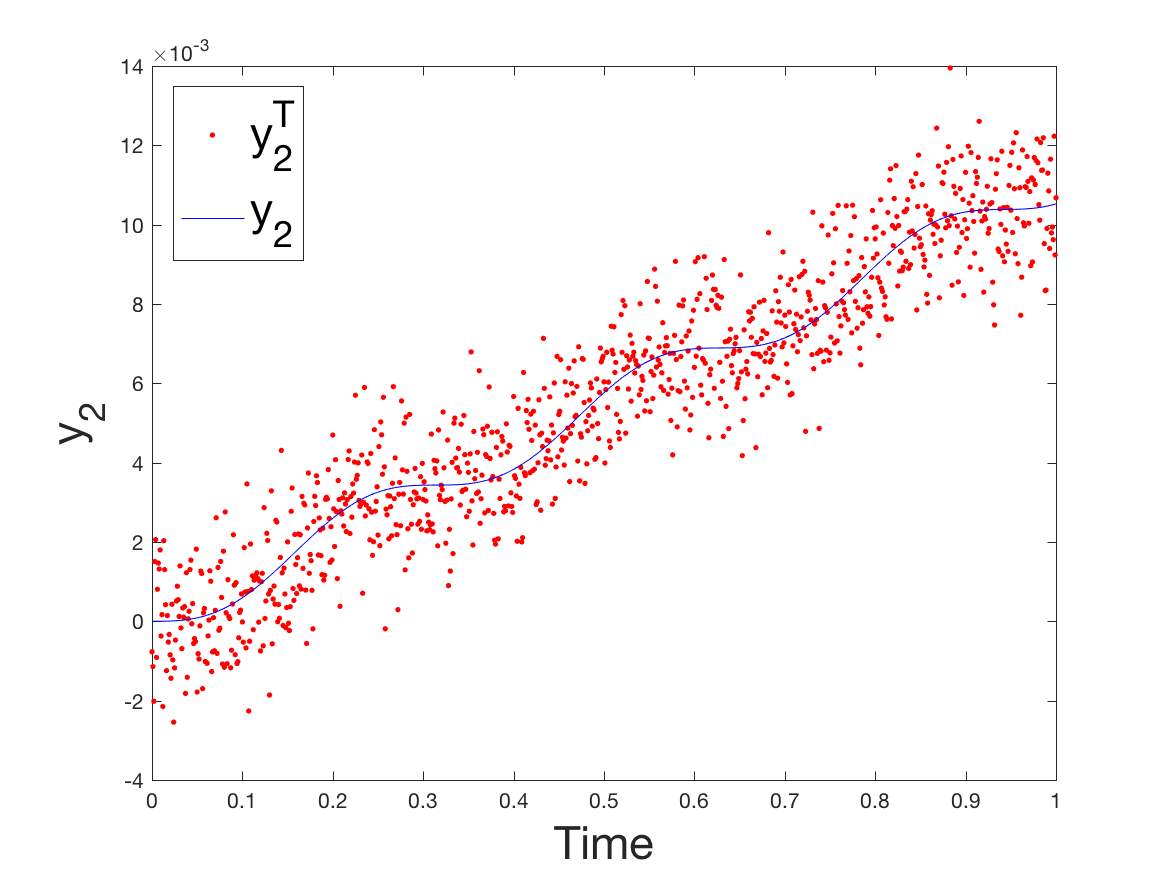}}&\rotatebox{-90}{10$\%$Noise}\\
\end{tabular}
\caption{Numerical experiments of the Linear model \eqref{linear1}-\eqref{linear2} Model fit with a noisy data (0\%,1\%,5\% and 10\% from top to bottom). Profile of the solution $\y_1$ (left) and $\y_2$ (right) with their respective targeted profile $\y_1^T$ and $\y_2^T$ respectively.}\label{LIN-plot}
\end{figure}

\begin{table}\begin{center}
\begin{tabular}{c|c|c|c}
\hline \text{Coef Variables} & $c_1$ &  $c_2$ &   $c_3$ \\ \hline
          \text{Actual Val.} & 1.190e-02 &3.520e-02 &2.200e-02 \\
\text{Noise-free Val.} &      1.186352e-01&        3.509264e-01&       2.224448e-01\\
1\% \text{Noise} &  1.197824e-01&     3.537101e-01&         2.189869e-01 \\
5\% \text{Noise} &      1.211243e-01&     3.523663e-01&     2.300441e-01 \\
10\% \text{Noise} &     1.250440e-01&    3.516971e-01&     2.197817e-01 \\\hline
\end{tabular}\caption{Convergence of the parameter estimation algorithm toward the solution given for the Linear model.}\label{LIN-data}
\end{center}\end{table}
The parameter estimations results for the linear model with variable matrix coefficients are reported in Table~\ref{LIN-data}, where the fit to the data is depicted in Fig.~\ref{LIN-plot}. 
\subsection{Second order parameter identification problem: Van Der Pol model}

Here, we consider a general second order ordinary differential equation as follow
\begin{equation}\label{2ndOrder}
 \ddot{\v} + \gamma(t)\dot{\v} + \lambda_\text{max}(t)\v= \f(t)
 \end{equation}
 Could be transformed to 
 \begin{equation}
\dot{ \begin{pmatrix}
 \v \\ \w
 \end{pmatrix}}+
 \begin{pmatrix}
 -\gamma(t){\lambda_\text{max}}^\prime(t) & \gamma(t)\lambda_\text{max}(t)\\
 -1/\gamma(t) & 1
 \end{pmatrix}
\begin{pmatrix}
 \v \\ \w
 \end{pmatrix}
 =
 \begin{pmatrix}
 \f(t)\gamma(t) \\ \f(t)
 \end{pmatrix}.
 \end{equation}
 Now, we can proceed with the optimization as described above. The example we are considering for the numerical illustration is the Van der Pol equation, which is known to model a non-conservative oscillator with non-liner damping. The Dynamics of this model are described as follow
\begin{equation}\label{venderpol}
 \ddot{y}(t) - \mu[1-y(t)^2] \dot{y}(t) + y(t)=0
 \end{equation}
The above dynamics can also be written as 
\begin{equation}\begin{array}{ccc}
\dot{v}(t)  &=& w(t)\\
\dot{w}(t)  &=& \mu[1-v(t)]w(t) -v(t).
\end{array}\end{equation}

The equation that governs the adjoint state variable writes 
\begin{equation}
\begin{array}{ccc}
-\dot{p}(t) = -(2\mu v(t) w(t) +1)q(t) \\
-\dot{q}(t) = p(t) + \mu(1-v^2(t))q(t)  
\end{array}
\end{equation}
and the gradient writes
\begin{equation}
\nabla \J(\mu) = \alpha \mu + (1-v^2(t))w(t) 
\end{equation}

Numerical results for the Vander Pol parameter identification is reported in Table~\ref{VDP-data} and the fit results toward the target solution is depicted in Fig.~\ref{VDP-plot}.

\begin{figure}[!htbp]
\begin{tabular}{ccc}
\raisebox{-.6\height}{\includegraphics[height=5cm,width=6cm]{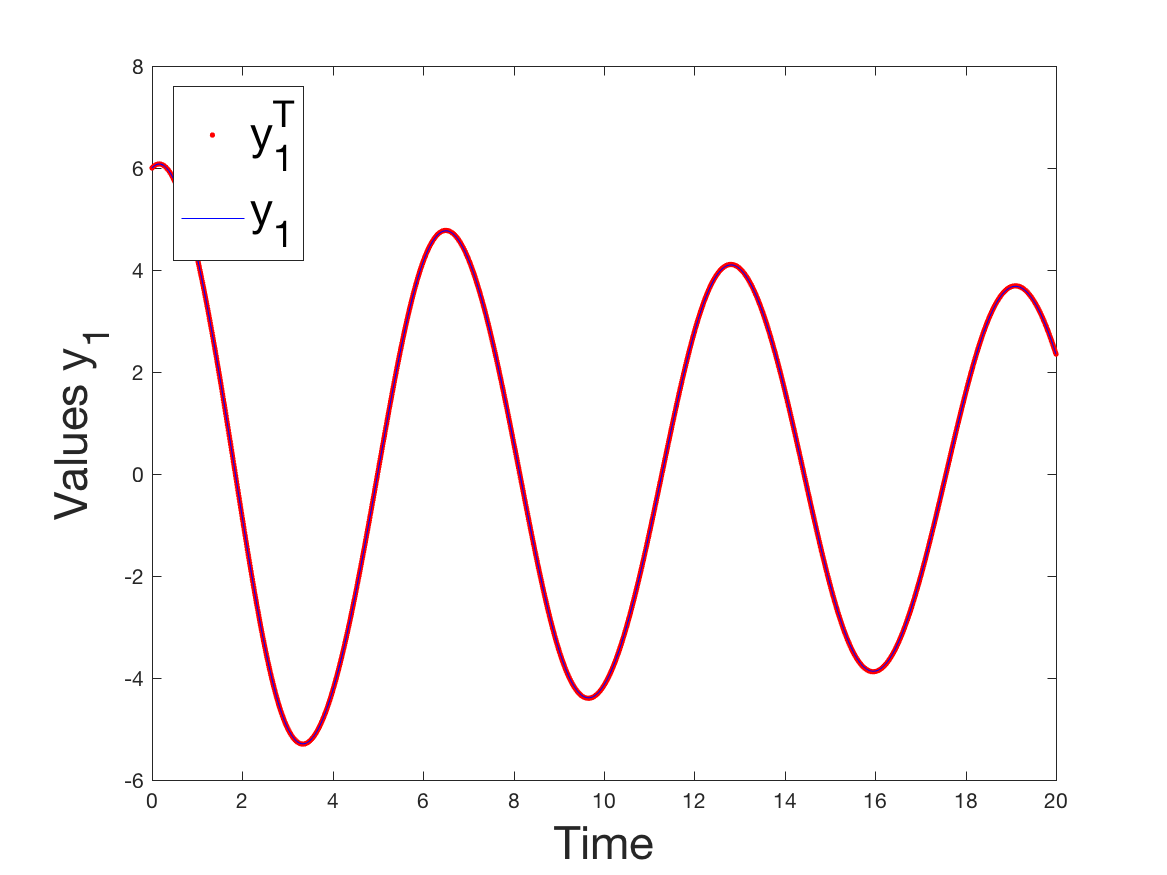}}&
\raisebox{-.6\height}{\includegraphics[height=5cm,width=6cm]{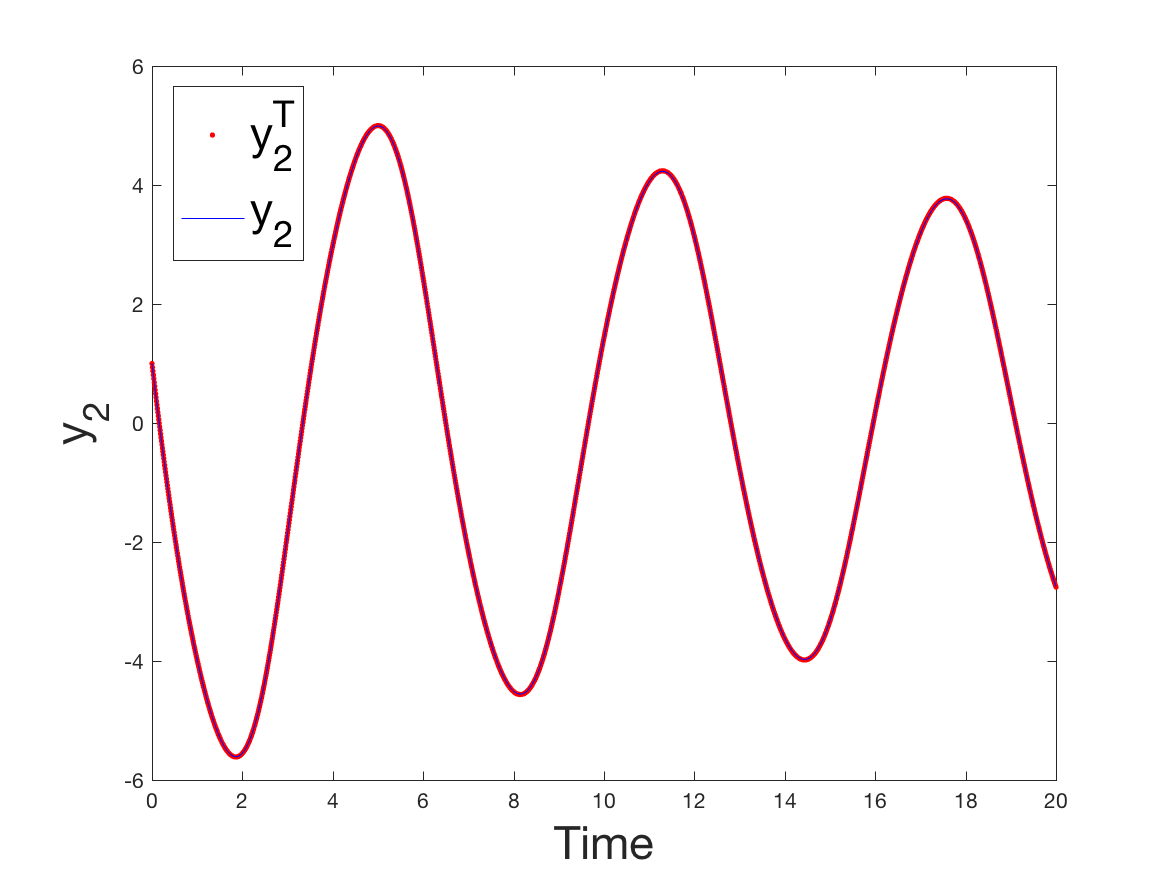}}&\rotatebox{-90}{Noise Free}\\
\raisebox{-.6\height}{\includegraphics[height=5cm,width=6cm]{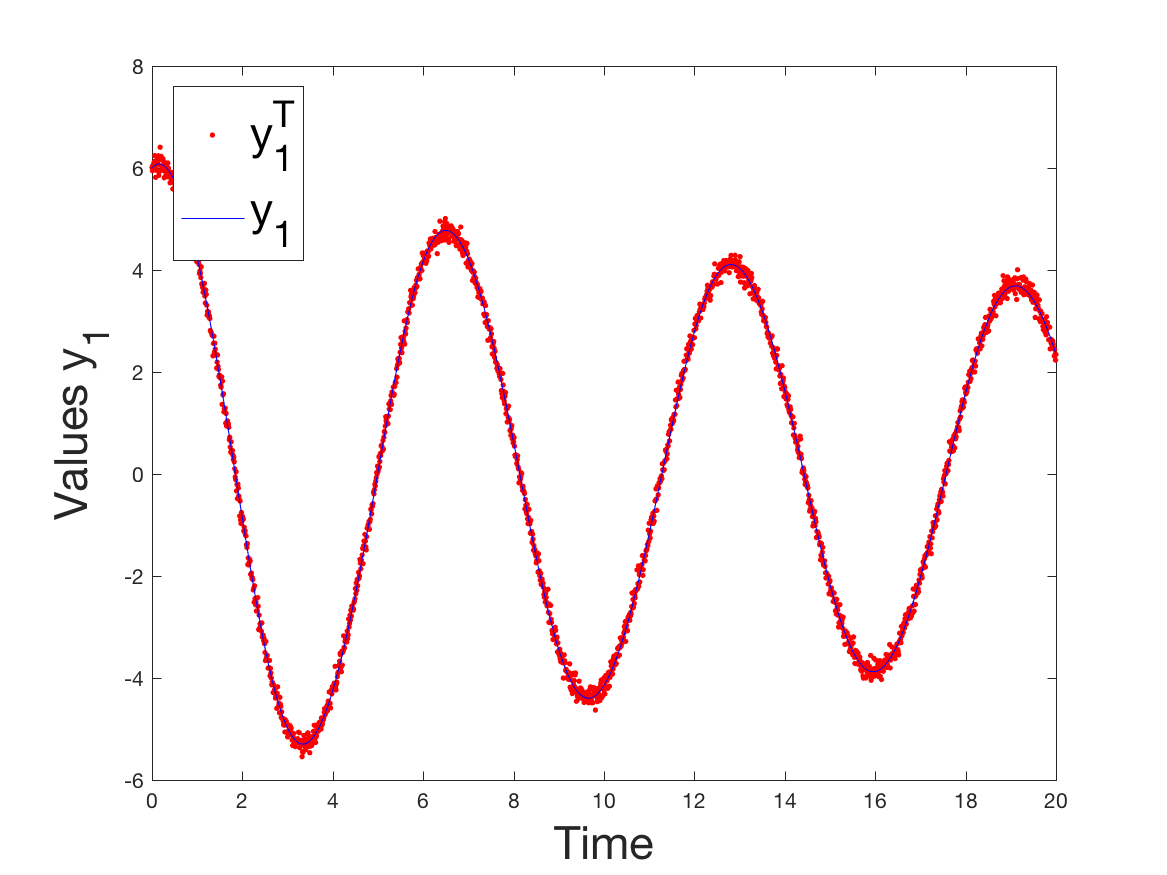}}&
\raisebox{-.6\height}{\includegraphics[height=5cm,width=6cm]{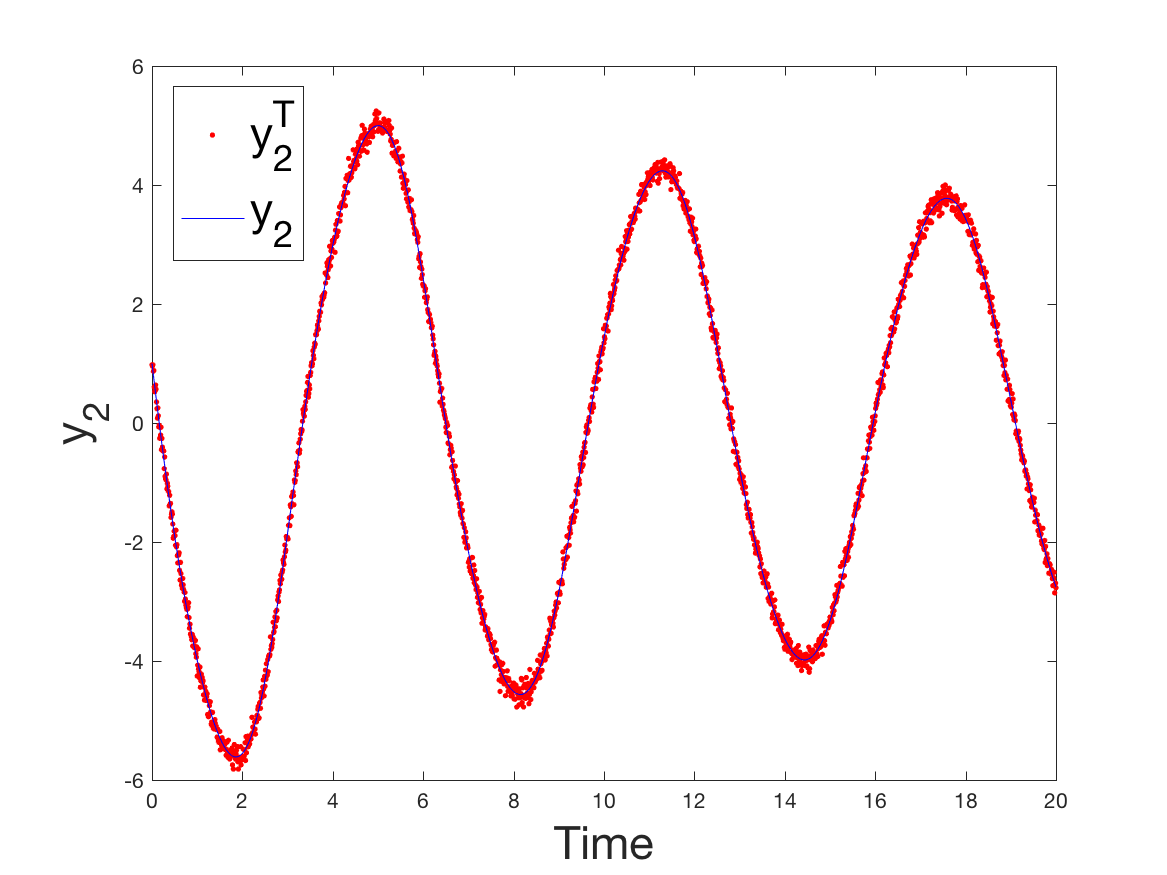}}&\rotatebox{-90}{1$\%$Noise}\\
\raisebox{-.6\height}{\includegraphics[height=5cm,width=6cm]{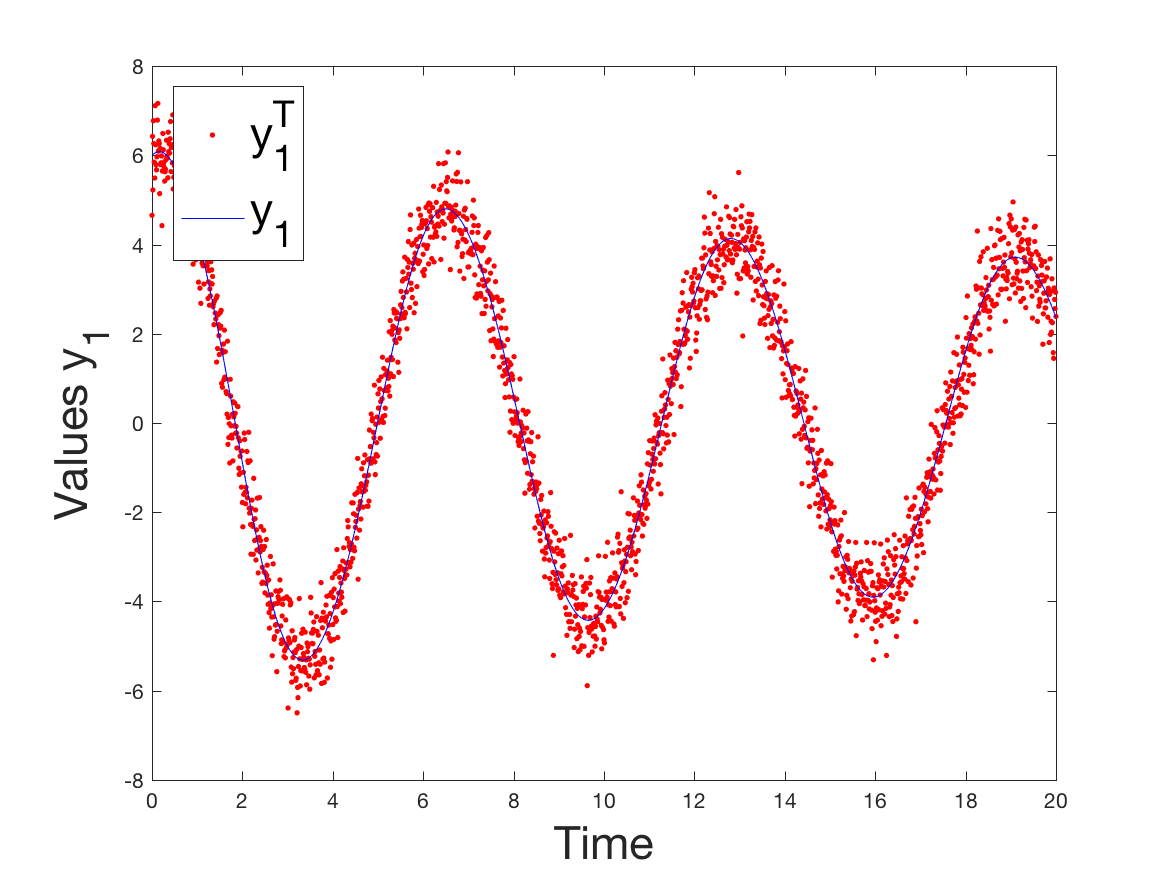}}&
\raisebox{-.6\height}{\includegraphics[height=5cm,width=6cm]{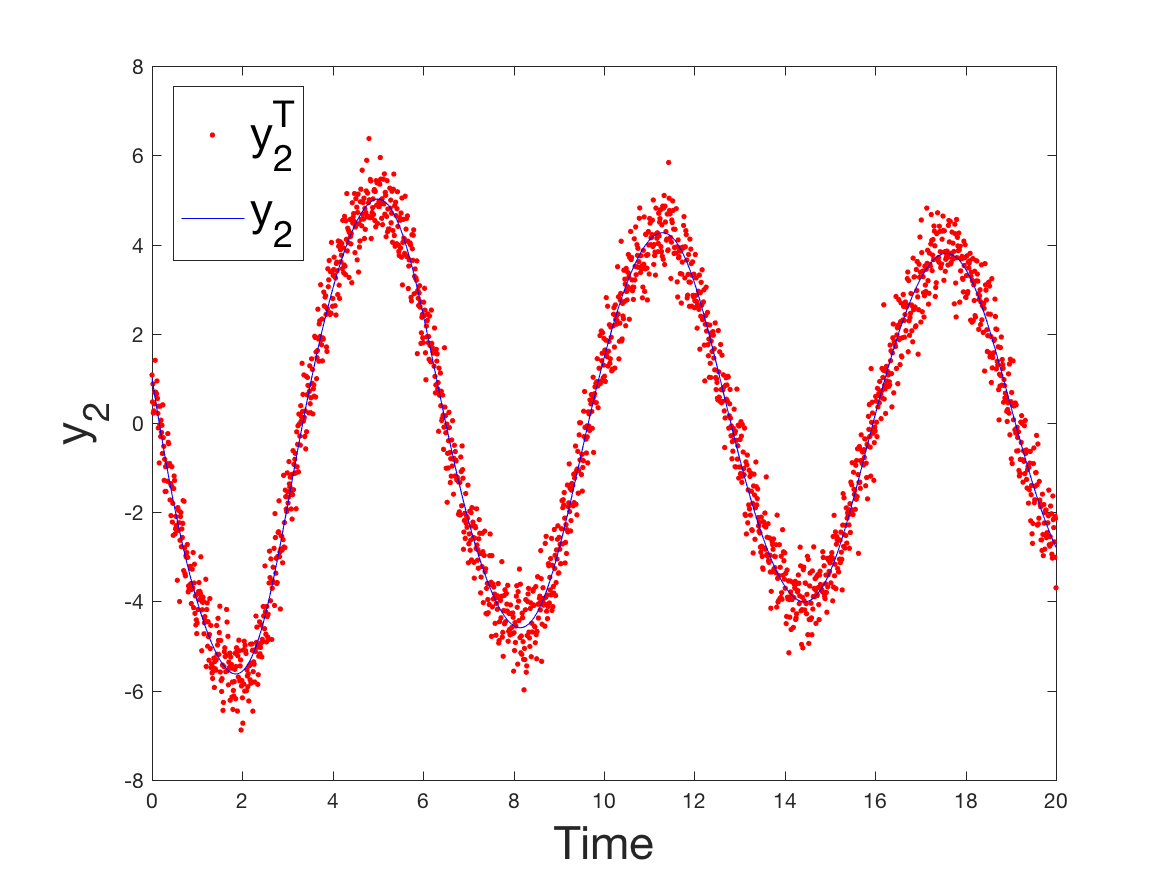}}&\rotatebox{-90}{5$\%$Noise}\\
\raisebox{-.6\height}{\includegraphics[height=5cm,width=6cm]{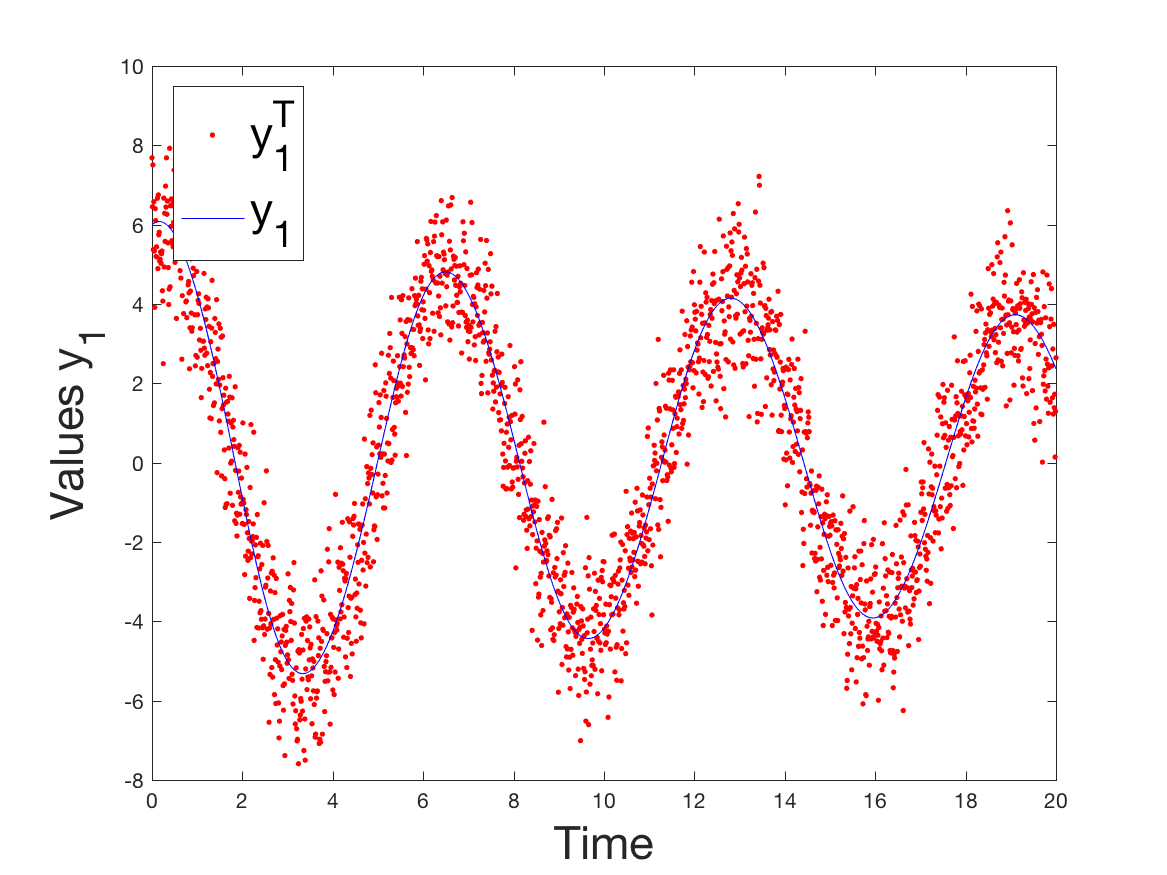}}&
\raisebox{-.6\height}{\includegraphics[height=5cm,width=6cm]{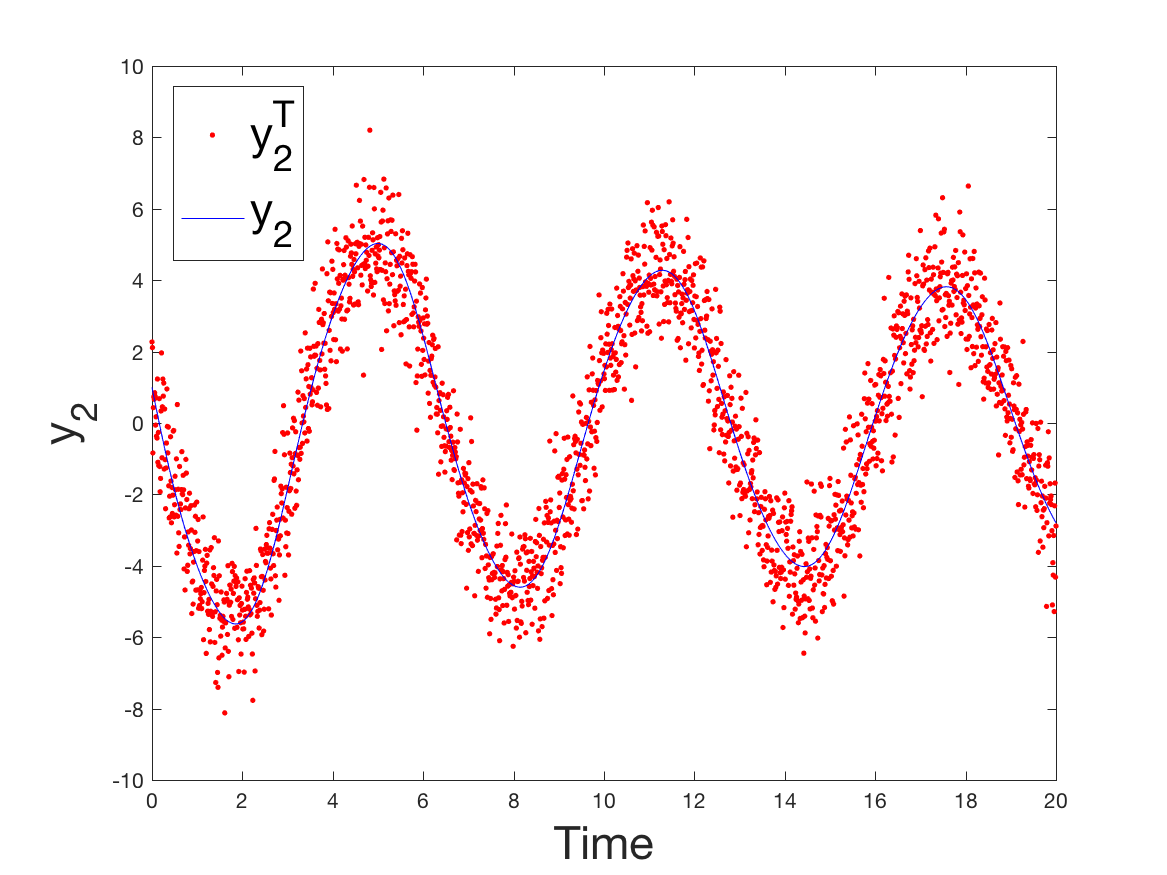}}&\rotatebox{-90}{10$\%$Noise}
\end{tabular}
\caption{Numerical experiments of the Vander Pol fit problem \eqref{venderpol} Model fit with a noisy data (0\%,1\%,5\% and 10\% from top to bottom). Profile of the solution $\y_1$ (left) and $\y_2$ (right) with their respective targeted profile $\y_1^T$ and $\y_2^T$ respectively.}\label{VDP-plot}
\end{figure}

\begin{table}[htbp]\begin{center}
\begin{tabular}{c|c}
\hline \text{Coef Variables} & $\mu$ \\ \hline   
 \text{Actual Val.} &  1.230e-02\\
\text{Noise-free Val.} & 1.239e-02 \\
1\% \text{Noise}  & 1.230e-02  \\
5\% \text{Noise}  &   1.199e-02\\
10\% \text{Noise}  & 1.185e-02 \\\hline
\end{tabular}\caption{Convergence of the parameter estimation algorithm toward the solution}\label{VDP-data}
\end{center}\end{table}

\subsection{Competing Species Model}

	In population dynamics modeling, Competing Species model, involves two interacting populations in some closed environment. We consider here, a two similar species competing for a limited food supply, without preying upon each other. 

\begin{equation}\label{CSM}\begin{array}{ccc}
\dot{v} &=& v\left( \zeta_{1} - \eta_{1} v - \theta_{1} w \right)\\
\dot{w} &=& w\left( \zeta_{2} - \eta_{2} w - \theta_{2}v \right)
\end{array}\end{equation}

where, $\zeta_{1}, \eta_{1}, \theta_{1}, \zeta_{2}, \eta_{2}$ and  $\theta_{2}$ are positive parameters to be identified in our  estimation problem. 

The optimality condition system writes as follows: In addition to state variables equations \eqref{CSM}, we have the adjoint state variables equations that read

\begin{equation}\label{AdjointCSM}
\begin{array}{ccc}
- \dot{p}(t) &=& \left( \xi_{1}-2\eta_{1} v(t) -\theta_{1} w \right) p(t)   - \theta_{2} w q(t) + v(t)-v^T(t)\\
- \dot{q}(t) &=& \left( \xi_{2}-2\eta_{2} v(t) -\theta_{2} w \right) q(t)   - \theta_{1} w p(t)+ v(t)-v^T(t).
\end{array} 
\end{equation}
supplemented with the gradient equation which writes 

\begin{equation*}
\g(\xi,\eta,\theta)=\alpha\begin{pmatrix}\xi_{1}\\ \xi_{2}\\ \eta_{1}\\ \eta_{2} \\ \theta_{1}\\\theta_{2}\end{pmatrix}
+ \begin{pmatrix}vp\\wq\\v^2p\\w^2q\\vwp\\wvq\end{pmatrix}
\end{equation*}

Numerical experiments related to the parameter identification of the Competing Species Model are reported in Table~\ref{CSM-data}. The fitting results are depicted in Fig.~\ref{CSM-plot}, which demonstrate convergence toward the target solution in the presence of noise. 

\begin{figure}[!htbp]
\begin{tabular}{ccc}
\raisebox{-.6\height}{\includegraphics[height=5cm,width=6cm]{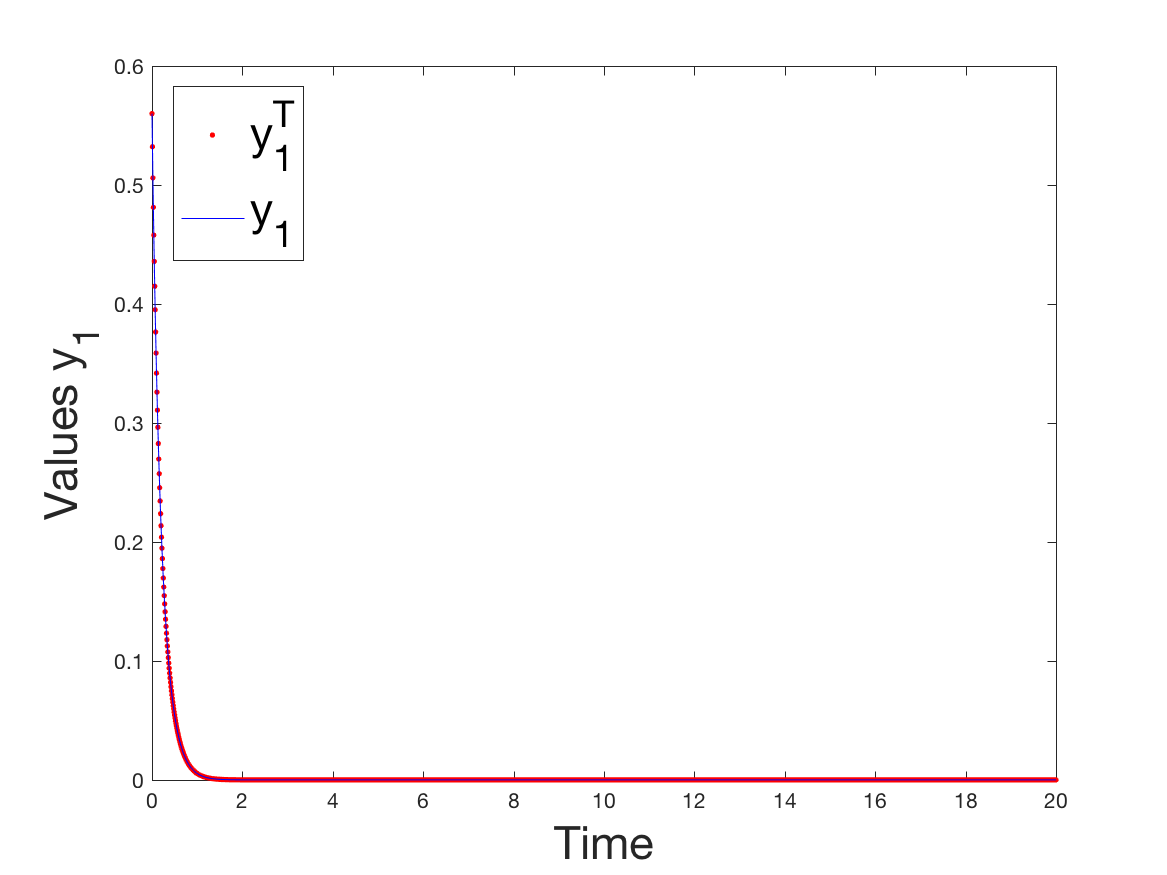}}&
\raisebox{-.6\height}{\includegraphics[height=5cm,width=6cm]{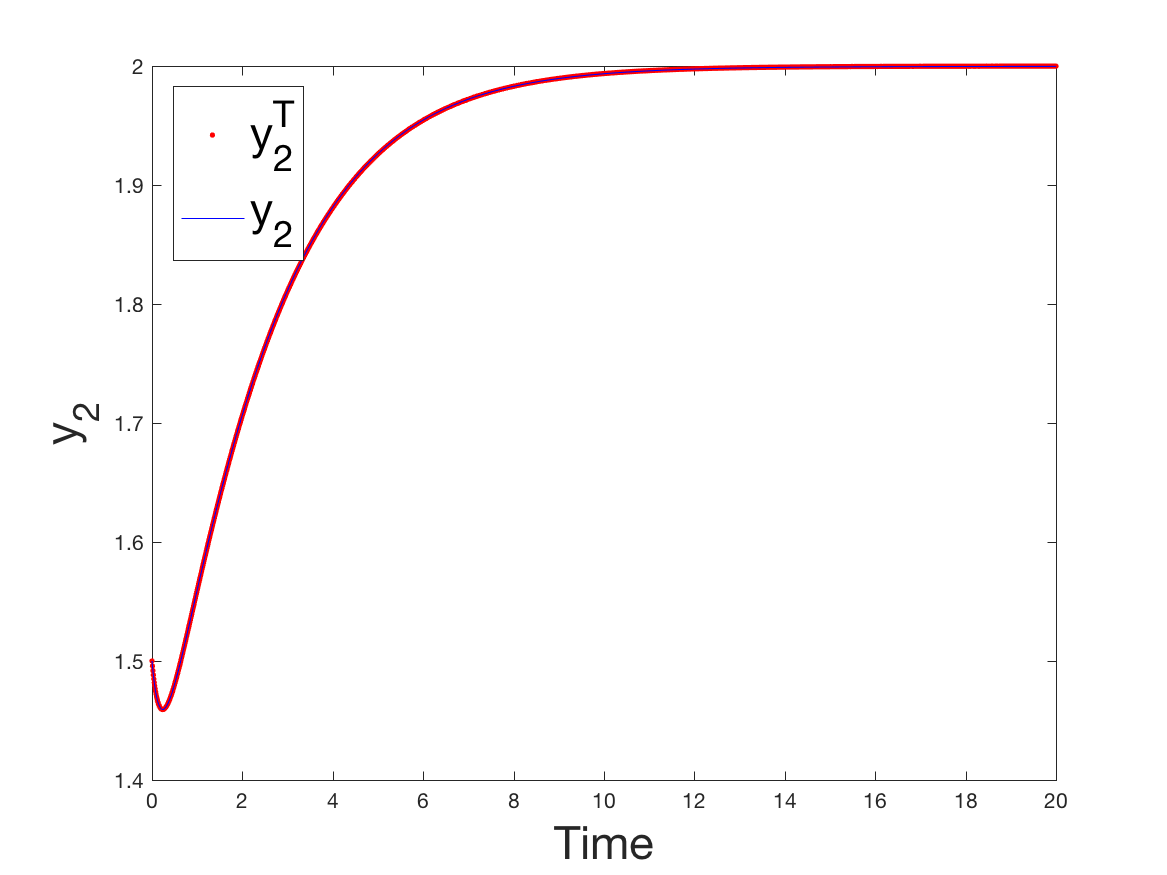}}&\rotatebox{-90}{Noise Free}\\
\raisebox{-.6\height}{\includegraphics[height=5cm,width=6cm]{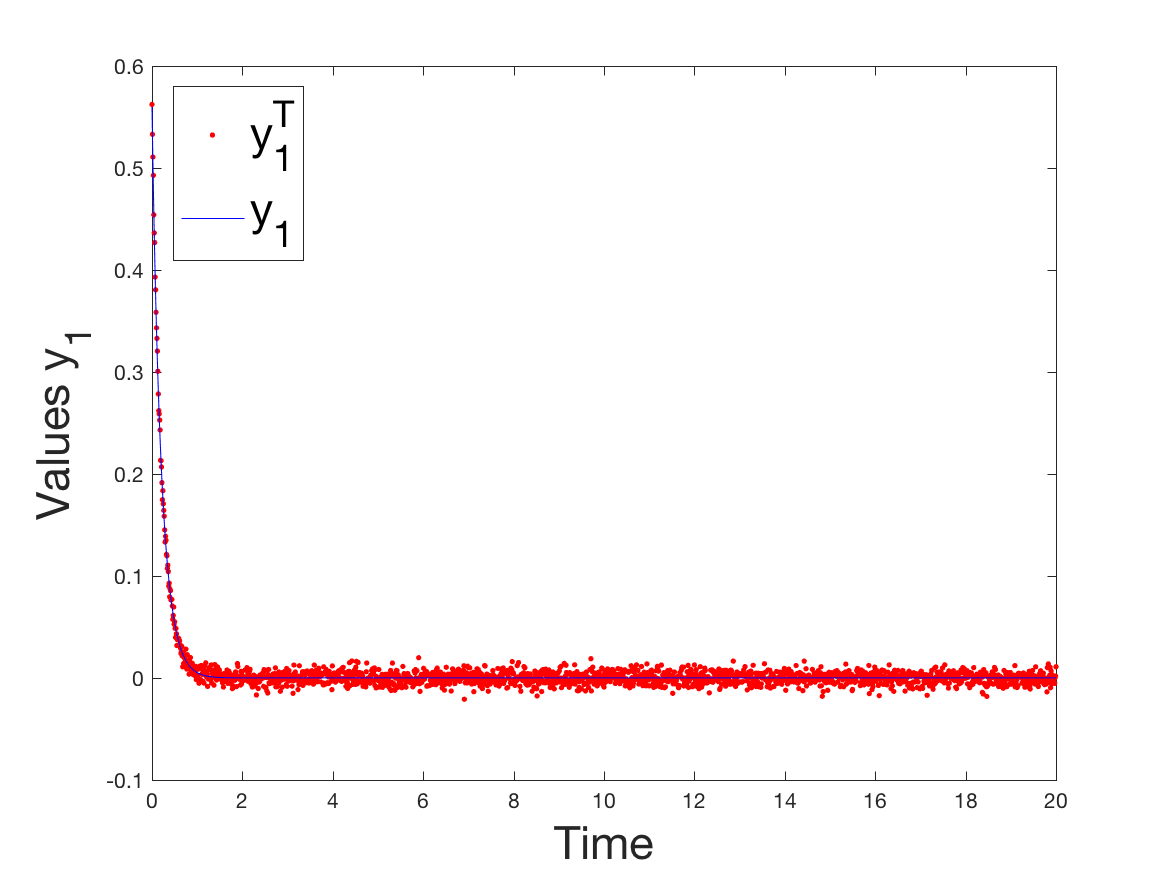}}&
\raisebox{-.6\height}{\includegraphics[height=5cm,width=6cm]{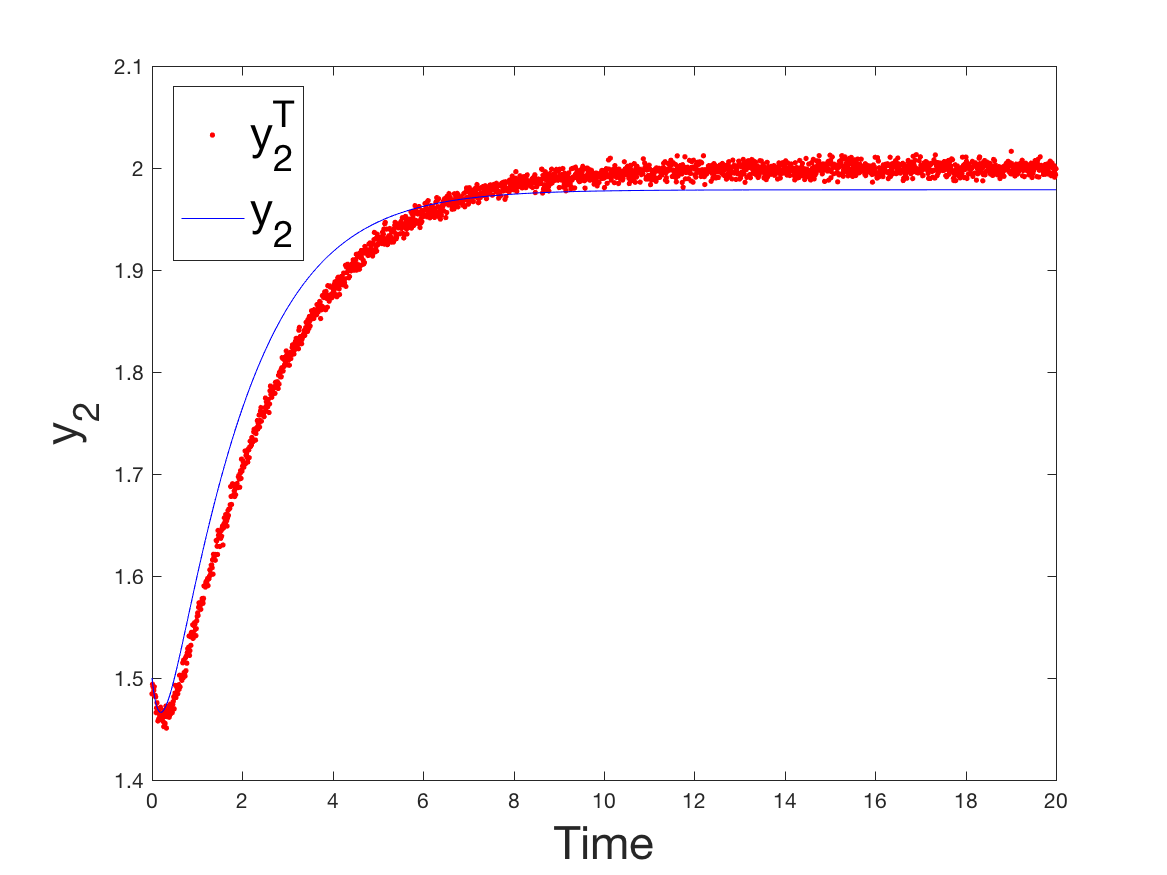}}&\rotatebox{-90}{1$\%$Noise}\\
\raisebox{-.6\height}{\includegraphics[height=5cm,width=6cm]{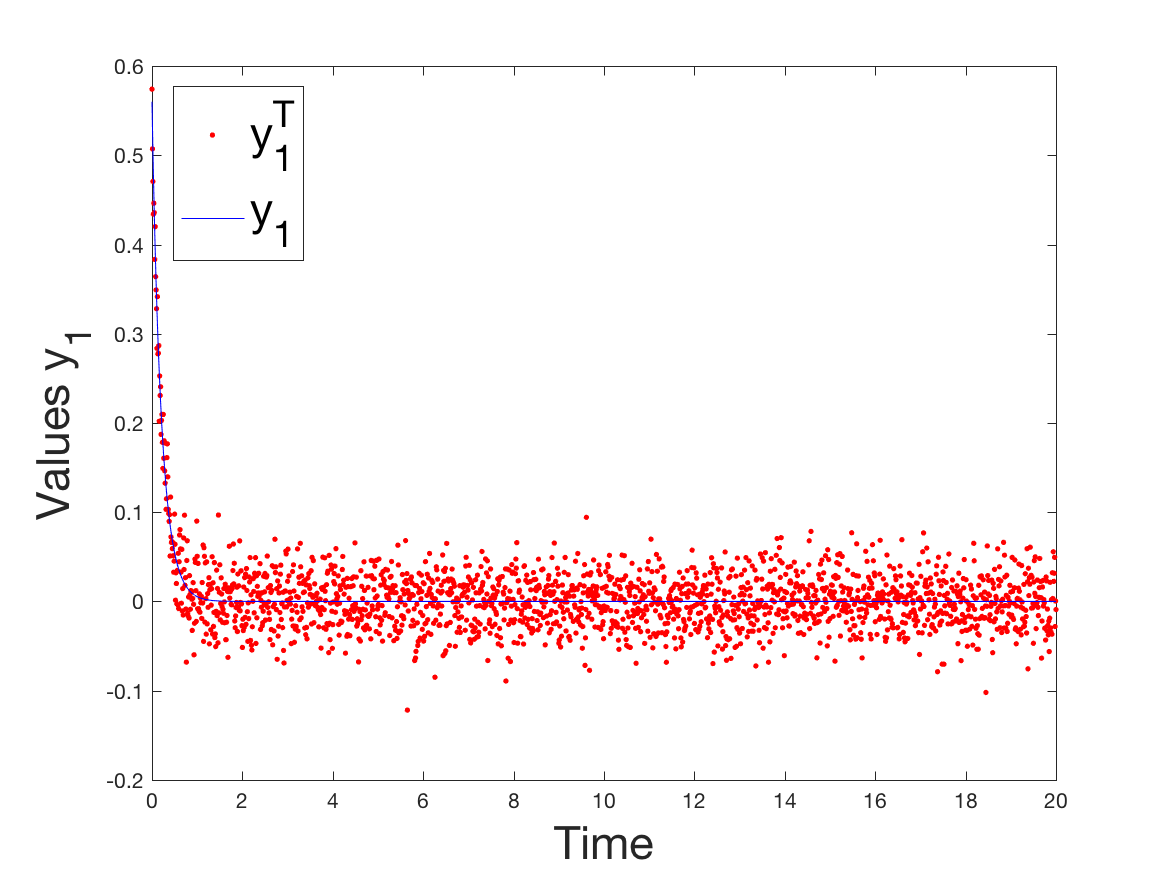}}&
\raisebox{-.6\height}{\includegraphics[height=5cm,width=6cm]{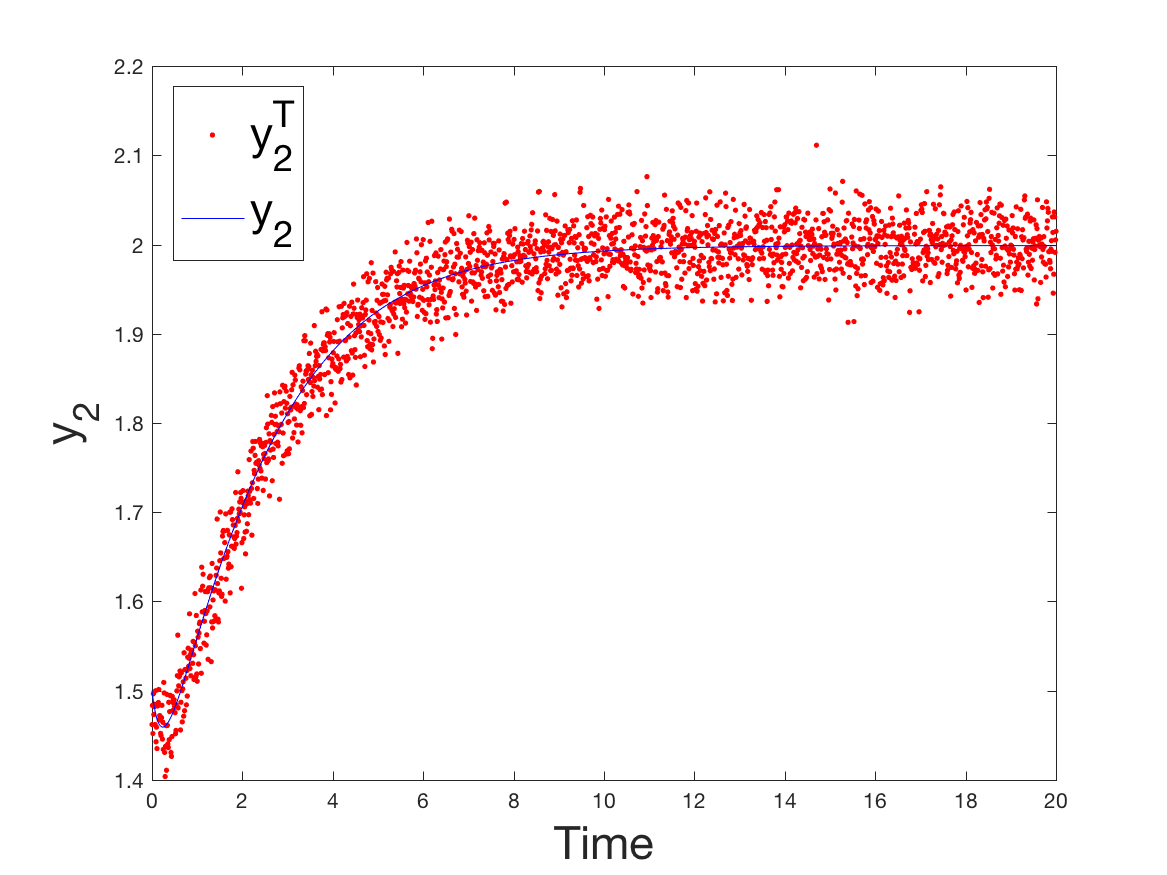}}&\rotatebox{-90}{5$\%$Noise}\\
\raisebox{-.6\height}{\includegraphics[height=5cm,width=6cm]{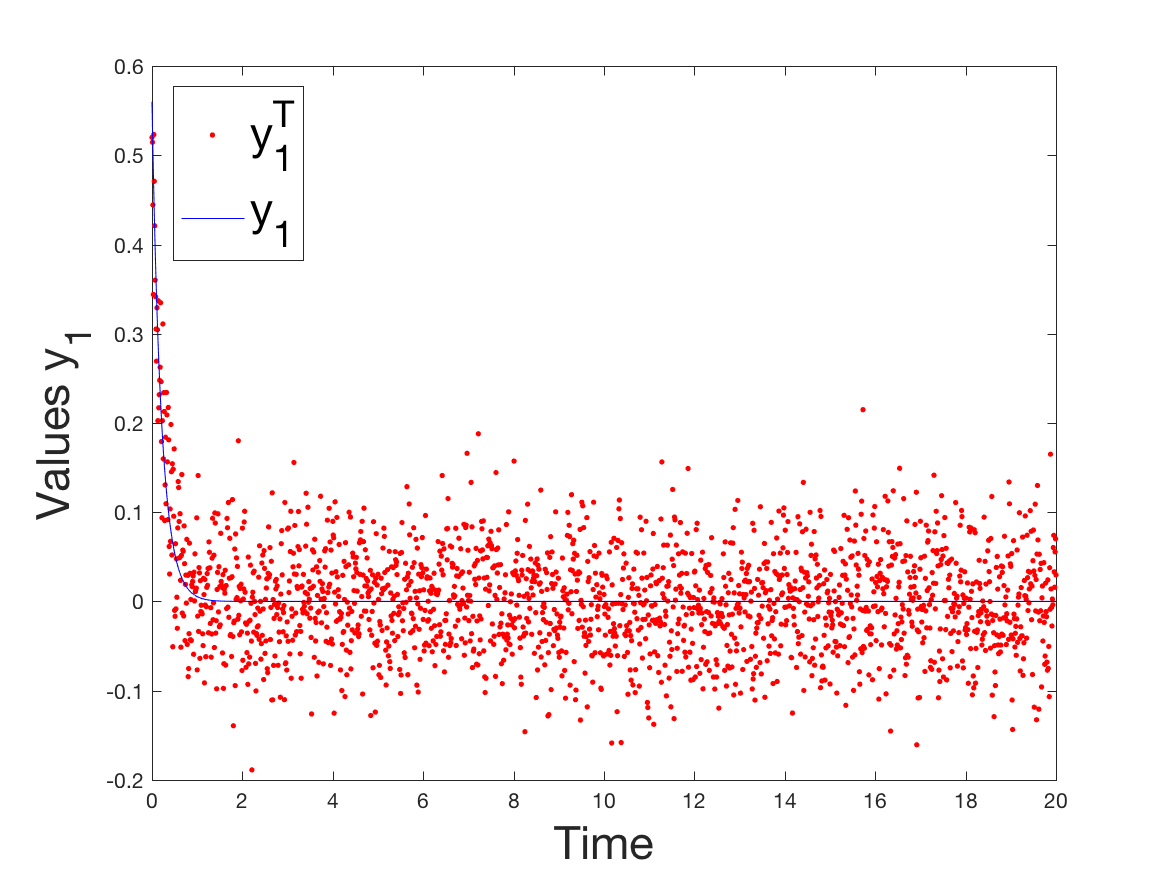}}&
\raisebox{-.6\height}{\includegraphics[height=5cm,width=6cm]{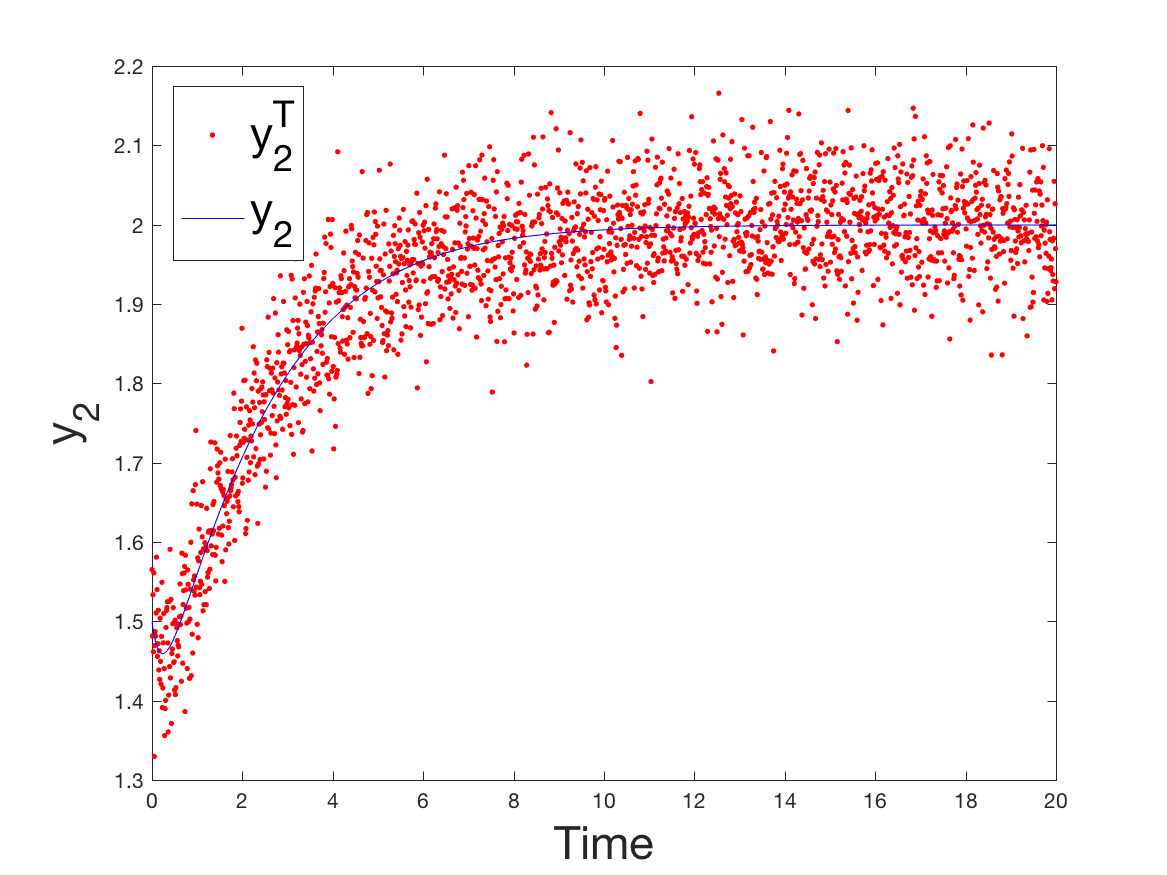}}&\rotatebox{-90}{10$\%$Noise}\\
\end{tabular}
\caption{Numerical experiments of the Species Compete Model fit problem \eqref{CSM} Model fit with a noisy data (0\%,1\%,5\% and 10\% from top to bottom). Profile of the solution $\y_1$ (left) and $\y_2$ (right) with their respective targeted profile $\y_1^T$ and $\y_2^T$ respectively.}\label{CSM-plot}
\end{figure}

\begin{table}[!h]
\begin{center}
\begin{tabular}{c|c|c|c|c|c|c}
\hline \text{Coef Variables} & $\zeta_{1}$&$ \eta_{1}$&$\theta_{1}$&$ \zeta_{2}$&$\eta_{2}$ & $\theta_{2}$\\ \hline
          \text{Actual Val.} &  4.0e-01 &      1.0e+00&      3.3e+00&      5.0e-01&      2.5e-01  &   7.5e-01\\
\text{Noise-free Val.} 
&     4.008846e-01&     1.000964e+00&     3.300283e+00&     5.014341e-01&     2.507478e-01&     7.507743e-01\\
1\% \text{Noise} 
&   4.806350e-01
&     1.1583938e+00
&    3.3306154e+00
&     6.710307e-01
&     3.391298e-01
&     7.997818e-01\\
5\% \text{Noise}  &4.009712e-01  &1.000412e+00  &3.300645e+00  &5.013874e-01  &2.508355e-01 &7.507349e-01\\
10\% \text{Noise}  
 &    4.006845e-01
 &    1.000858e+00
  &   3.300348e+00
  &   5.012614e-01
  &   2.506655e-01
  &   7.502489e-01
  \\\hline
\end{tabular}\caption{Convergence of the parameter estimation algorithm toward the solution of Species Compte Nonlinear Model}\label{CSM-data}
\end{center}\end{table}

\section{Summary and conclusion}\label{Conc}
We considered in this paper a non-linear parameter estimation problem for a class of linear dynamical model. We proved, under necessary conditions on the smoothness of the handled problem, the linear convergence of the formulated optimal control problem using a state-constrained nonlinear least-square minimization via the classical steepest descent method.  Besides, we proved linear convergence rate for a class of nonlinear conjugate gradient method. Our analysis differs from the literature where previous attempts employ contradiction evidence to prove global convergence. We think that our proof's steps will help in a further understanding of the convergence properties of the nonlinear conjugate gradient.

    Numerical evidence has been reported to show the effectiveness of the convergence. We don't claim that our numerical experiments prove convergence to the absolute minimum, rather than showing comfortable convergence toward stationary point with the help of a sampling procedure. It is noticed here that the later takes considerable wall-time to generate good initial guess depending on the chosen sampling criteria.
 In future work, we shall investigate global optimizer search technique combined with the NCG in order to subjugate the local convergence limitations.  

\bibliography{bibParam}
\bibliographystyle{plain}

\end{document}